\documentclass[a4paper,12pt]{article}
\usepackage{longtable} 
\usepackage{amsmath,amssymb,amsthm}
\usepackage{enumerate}
\usepackage{array}
\newcolumntype{C}[1]{>{\centering\arraybackslash}m{#1}}
\newcolumntype{L}[1]{>{\raggedright\arraybackslash}m{#1}}

\newtheorem{theorem}{Theorem}[section]
\newtheorem{lemma}[theorem]{Lemma}
\newtheorem{proposition}[theorem]{Proposition}
\newtheorem{corollary}[theorem]{Corollary}
\theoremstyle{definition}
\newtheorem{definition}[theorem]{Definition}

\theoremstyle{remark}
\newtheorem{remark}[theorem]{Remark}

\newtheorem{conventions}[theorem]{Conventions}

\numberwithin{equation}{section}
\long\def\symbolfootnote[#1]#2{\begingroup%
\def\thefootnote{\fnsymbol{footnote}}\footnote[#1]{#2}\endgroup} 
\newcommand{\Ann}[1]{{#1}^0}
\newcommand{\e}{\mathrm{e}}

\newcommand{\hook}{\lrcorner \,}
\newcommand{\id}{\mathrm{id}}

\newcommand{\vol}{\mathrm{vol}}

\newcommand{\diag}{\mathrm{diag}} 
 
\newcommand{\pr}{\mathrm{pr}}

\newcommand{\bR}{\mathbb{R}} 
\newcommand{\bF}{\mathbb{F}}
\newcommand{\rk}{\mathrm{rk}}

\renewcommand{\Im}{\mathrm{Im}}

\renewcommand{\dim}[1]{\mathrm{dim}(#1)}
\newcommand{\spa}[1]{\mathrm{span}(#1)}

\newcommand{\g}{\mathfrak{g}} 
\newcommand{\h}{\mathfrak{h}}
\newcommand{\uf}{\mathfrak{u}} 
\newcommand{\af}{\mathfrak{a}} 
\newcommand{\GL}{\mathrm{GL}} 

\newcommand{\SO}{\mathrm{SO}}
 
\newcommand{\SU}{\mathrm{SU}}
\newcommand{\G}{\mathrm{G}}
\newcommand{\Spin}{\mathrm{Spin}}
\newcommand{\so}{\mathfrak{so}} 

\newcommand{\op}{\oplus} 

\newcommand{\tr}{\mathrm{tr}}

\renewcommand{\L}{\Lambda}

\begin{document}
\title{Cocalibrated $\G_2$-structures on products of four- and three-dimensional Lie groups}

\author{
Marco Freibert \\
\small Fachbereich Mathematik\\
\small Universit\"at Hamburg \\
\small Bundesstra{\ss}e 55, 
D-20146 Hamburg, Germany \\
\small \texttt{freibert@math.uni-hamburg.de}
}
\maketitle

\begin{abstract}
Cocalibrated $\G_2$-structures are structures naturally induced on hypersurfaces in $\Spin(7)$-manifolds. Conversely, one may start with a seven-dimensional manifold $M$ endowed with a cocalibrated $\G_2$-structure and construct via the Hitchin flow a $\Spin(7)$-manifold which contains $M$ as a hypersurface. In this article, we consider left-invariant cocalibrated $\G_2$-structures on Lie groups $\G$ which are a direct product $\G=\G_4\times \G_3$ of a four-dimensional Lie group $\G_4$ and a three-dimensional Lie group $\G_3$. We achieve a full classification of the Lie groups $\G=\G_4\times \G_3$ which admit a left-invariant cocalibrated $\G_2$-structure.
\end{abstract}

\maketitle

\section{Introduction}
\label{intro}
A $\G_2$-structure on a seven-dimensional manifold $M$ is a three-form $\varphi\in \Omega^3 M$ on $M$ with pointwise stabilizer conjugated to $\G_2\subseteq \SO(7)$.
Such a three-form $\varphi\in \Omega^3 M$ naturally induces a Riemannian metric, an orientation and so a Hodge star operator $\star_{\varphi}:\Omega^{*} M \rightarrow \Omega^{*} M$ on $M$. We call $\varphi$ \emph{cocalibrated} if
\begin{equation*}
d\star_{\varphi} \varphi=0.
\end{equation*}
Interest on cocalibrated $\G_2$-structures arises from different sources. First of all, they appear as one class of $\G_2$-structures in the Fern\'{a}ndez-Gray classification \cite{FG} of $\G_2$-structures by their intrinsic torsion. Secondly, they naturally appear in the context of Strominger's equations \cite{Str} in type II string theory, cf. e.g \cite{FI} and \cite{FIUV}. Moreover, any hypersurface in an eight-dimensional Riemannian manifold with holonomy contained in the exceptional holonomy group $\Spin(7)$ naturally carries a cocalibrated $\G_2$-structure \cite{MC}. Most importantly, also the converse relation between seven-dimensional manifolds with cocalibrated $\G_2$-structures and eight-dimensional Riemannian manifold with holonomy contained in $\Spin(7)$ holds. Given a seven-dimensional real-analytic manifold $M$ with real-analytic cocalibrated $\G_2$-structure, one may construct an eight-dimensional $\Spin(7)$-manifold containing $M$ as a hypersurface by solving a system of time-dependent partial differential equations, the so-called Hitchin's flow equations, cf. \cite{Hi}, \cite{CLSS}.

Hence, one is interested in constructing examples of real-analytic cocalibrated $\G_2$-struc\-ture and, as a first step, identifying the real-analytic manifolds which admit real-analytic cocalibrated $\G_2$-structures at all. The latter problem has been addressed e.g. in \cite{R}, where the compact homogeneous spaces admitting homogeneous cocalibrated $\G_2$-structures are determined. 
In \cite{F}, the author classified the seven-dimensional almost Abelian Lie groups possessing a left-invariant cocalibrated $\G_2$-structure.

In this paper we look again at left-invariant cocalibrated $\G_2$-structures on Lie groups $\G$, namely on those $\G$ which are a direct product of a three-dimensional Lie group $\G_3$ and a four-dimensional Lie group $\G_4$. We classify which of these Lie groups admit left-invariant cocalibrated
$\G_2$-structures.

Identifying as usual left-invariant $k$-forms on the Lie group with $k$-forms on the Lie algebra and introducing a differential on $\L^{*} \g^*$
by this identification, we may speak of \emph{cocalibrated $\G_2$-structures on a seven-dimensional Lie algebra} and these forms are in one-to-one correspondence to left-invariant cocalibrated $\G_2$-structures on each corresponding Lie group. Our main result can now be formulated as follows,
where we refer the reader for the names of the appearing Lie algebras to the Tables \ref{table:3d} and \ref{table:4d}.
\begin{theorem}\label{th:main}
Let $\g=\g_4\op \g_3$ be a seven-dimensional Lie algebra which is the Lie algebra direct sum of a
four-dimensional Lie algebra $\g_4$ and of a three-dimensional Lie algebra $\g_3$. Then $\g$ admits
a cocalibrated $\G_2$-structure if and only if one of the following four conditions is fulfilled:
\begin{enumerate}
\item
$\g_4$ is not unimodular, $\g_3$ is unimodular and $h^1(\g_4)+h^1(\uf)-h^2(\g_4)+h^2(\g_3)\leq 4$, where $\uf$ is the unimodular kernel of $\g_4$.
\item
$\g_4$ is unimodular, $\g_3$ is unimodular and at least one of the following conditions is true:
\begin{enumerate}
\item
$\g_3\in \{\so(3),\so(2,1)\}$
\item
$\g_4=\h\op \bR$ for a three-dimensional unimodular Lie algebra $\h$.
\item
$\g\in\{A_{4,1}\op e(2), A_{4,1}\op e(1,1), A_{4,8}\op e(1,1)\}$.
\end{enumerate}
\item
$\g_4$ is unimodular, $\g_3$ is not unimodular and at least one of the following conditions is true:
\begin{enumerate}
\item
$\g_4$ is almost Abelian, $\g_4\notin\{\bR^4,\h_3\op \bR\}$ and $\g_3=\mathfrak{r}_2\op \bR$.
\item
$[\g_4,\g_4]\in\{\h_3, \so(3),\so(2,1)\}$.
\end{enumerate}
\item
$\g_4$ is not unimodular, $\g_3$ is not unimodular and at least one of the following conditions is true:
\begin{enumerate}
\item
The unimodular kernel $\uf$ of $\g_4$ is isomorphic to $e(2)$ or $e(1,1)$.
\item
$\g=A_{4,9}^{-\frac{1}{2}}\op\mathfrak{r}_2\op \bR$. 
\item
The unimodular kernel $\uf$ of $\g_4$ is isomorphic to $\h_3$, $\g_3\neq \mathfrak{r}_2\op \bR$ and\\
$\g\notin\left\{A_{4,9}^1\oplus \mathfrak{r}_{3,\mu},A_{4,9}^{\alpha}\oplus \mathfrak{r}_{3,1}\left|\mu\in \left[-1/3,0\right),\,\alpha\in \left(-1,-1/3\right]\right.\right\}$.
\end{enumerate}
\end{enumerate}
\end{theorem}
For the proof of Theorem \ref{th:main} we use as in \cite{F} the algebraic invariants introduced by Westwick \cite{W}. In contrast to \cite{F},
these algebraic invariants only lead to obstructions. The construction of cocalibrated $\G_2$-structures relies on the following two properties of $\G_2$-structures. Firstly, from a decomposition $\g=V_4\oplus V_3$ of $\g$ into a four-dimensional subspace $V_4$ and a three-dimensional subspace $V_3$ and certain two-forms on $V_4$ and $V_3$ one can build the Hodge dual of a $\G_2$-structure. Note that in the concrete applications later these subspaces may not always coincide with $\g_4$ and $\g_3$. Secondly, we use the openness of the orbit of all Hodge duals. Therefore, we write down the Hodge dual $\Psi\in \L^4 \g^*$ of a $\G_2$-structure ``well-adapted'' to the structure of the Lie algebra $\g$, add some term $\Phi\in \L^4 \g^*$ such that $\Psi+\Phi$ is closed and rescale $\Psi$ and $\Phi$ such that the sum stays closed and $\Phi$ gets small in comparison to $\Psi$. Then $\Psi+\Phi$ is the Hodge dual of a cocalibrated $\G_2$-structure.

The work is organized as follows:
In Section \ref{sec:prelim}, we deal with preliminaries on $\G_2$-structures, four- and three-dimensional Lie algebras and the mentioned algebraic invariants. We begin in Subsection \ref{subsec:G2onvs} by recalling the definition and basic properties of a $\G_2$-structure on a seven-dimensional vector space. Moreover, we show that the orbit of all Hodge duals of such structures is ``uniformly'' open in a sense made precise in that subsection. In the following subsection, we expand our definition to $\G_2$-structures on manifolds and introduce cocalibrated $\G_2$-structures on Lie algebras. Subsections \ref{subsec:3dLA} and \ref{subsec:4dLA} are devoted to recalling basic facts about three-dimensional and four-dimensional Lie algebras. In Subsection \ref{subsec:alginv}, we recall the algebraic invariants for $k$-vectors introduced partly by Westwick \cite{W} and the values of these invariants for certain $k$-forms associated to $\G_2$-structures obtained in \cite{W} and \cite{F}. We investigate in Subsection  \ref{subsec:constructionG2} under which circumstances a subspace of the space of all two-forms on a four-dimensional vector space consists entirely of non-degenerate two-forms and how one can build from such two-forms the Hodge dual of a $\G_2$-structure on a seven-dimensional vector space.

In Section \ref{sec:class}, we give the classification. For that purpose we use in Subsection \ref{subsec:existence} the ``uniform'' openness of the orbit of all Hodge duals to show that, under certain assumptions, one may deform a given $\G_2$-structure on a seven-dimensional manifold in a particular way to obtain a one-parameter family of cocalibrated $\G_2$-structures on $M$. We apply this result to our situation, namely $\G_2$-structures on Lie algebras which are direct sums of a four-dimensional and a three-dimensional Lie algebra, to get existence results for certain classes of such Lie algebras. In Subsection \ref{subsec:obstructions} we use the algebraic invariants to obtain obstructions to the existence of cocalibrated $\G_2$-structures on the Lie algebras in question and exclude such structures for large classes. In the Subsections \ref{subsec:firstcase} -  \ref{subsec:fourthcase}, we apply the results of the Subsections \ref{subsec:existence} and \ref{subsec:obstructions} to the direct sums $\g=\g_4\op \g_3$ and deal separately with the four cases which naturally appear by distinguishing whether $\g_4$ or $\g_3$ is unimodular or not.
\section{Preliminaries}\label{sec:prelim}
Throughout this article, we use the following conventions
\begin{conventions}\label{conventions}
All considered vector spaces, Lie algebras, etc., are real and finite-dimen\-sional. If $V$ is a vector space and $A$ is a subset of $V$, we denote by $\Ann{A}:=\{\alpha\in V^*|\alpha(a)=0\, \forall a\in A\}$ the annihilator of $A$ in $V$. If $V=W\op U$ as $\bF$-vector spaces and $\pi_W:V\rightarrow W$ is the projection onto $W$ along $U$, then $\pi_W^*:\L^* W^*\rightarrow \L^* V^*$ is injective. The image of $\pi_W^*$ is $\L^* \Ann{U}$. We use this to identify $\L^* \Ann{U}$ with $\L^* W^*$. If $\g=\uf\op U$ is a real finite-dimensional Lie algebra which is the vector space direct sum of an ideal $\uf$ in $\g$ and a vector subspace $U\subseteq \g$, then the above injection also identifies the cochain complexes $(\L^* \Ann{U}, \pi_{\L^*\Ann{U}} \circ d_{\g}|_{\L^* \Ann{U}})$ and $(\L^* \uf^*,d_{\uf})$, where $\pi_{\L^*\Ann{U}}: \L^*\g^*\rightarrow \L^* \Ann{U}$ is the projection onto $\L^* \Ann{U}$ along $ \Ann{\uf}\wedge \L^* \g^*$. Using this identification, we write $d_{\uf}$ instead of $\pi_{\L^*\Ann{U}} \circ d_{\g}|_{\L^* \Ann{U}}$. Note that if $U$ is also an ideal in $\g$ and $\g=\uf\op U$ is a Lie algebra direct sum, then $\pi_{\L^*\Ann{U}} \circ d_{\g}|_{\L^* \Ann{U}}=d_{\g}|_{\L^* \uf^*}=d_{\uf}$ in our identification. In this case we omit the index and simply write $d$.
\end{conventions}
\subsection{$\G_2$-structures on vector spaces} \label{subsec:G2onvs}
We give a short introduction into $\G_2$-structures on vector spaces. More thorough introductions may be found in \cite{Br1} and in \cite{F}.
\begin{definition}
Let $V$ be a seven-dimensional vector space.
A \emph{$\G_2$-structure} on $V$ is a three-form $\varphi\in \L^3 V^*$
for which there exists a basis $e_1,\ldots,e_7$ of $V$ with
\begin{equation}\label{eq:G2structure}
\varphi=e^{127}+e^{347}+e^{567}+e^{135}-e^{146}-e^{236}-e^{245}
\end{equation}
Thereby, $e^1,\ldots,e^7\in V^*$ denotes the dual basis of $e_1,\ldots,e_7$.
We call the seven-tuple $(e_1,\ldots e_7)\in V^7$ an \emph{adapted basis}
for the $\G_2$-structure $\varphi$.
\end{definition}
\begin{remark}
All $\G_2$-structures lie in one orbit under the natural action
of $\GL(V)$ on $\L^3 V^*$. The isotropy group of a $\G_2$-structure in $\GL(V)$ under this action
is isomorphic to $\G_2$, which is in our context the simply-connected compact real form of the
complex simple Lie group $(\G_2)_{\mathbb{C}}$. Since $\dim{\GL(V)}=49$, $\dim{\G_2}=14$ and $\dim{\L^3 V^*}=35$,
the orbit is open, i.e. a $\G_2$-structure is a \emph{stable form} \cite{Hi}.
Note that there is another open orbit in $\L^3 V^*$ whose stabilizer is $\G_2^*$, the
split real form of $(\G_2)_{\mathbb{C}}$ with $\pi_1(\G_2^*)=\mathbb{Z}_2$ \cite{Br1}.
\end{remark}
Since $\G_2\subseteq \SO(7)$, a $\G_2$-structure induces a Euclidean metric and an orientation
on $V$ as follows \cite{CLSS}:
\begin{lemma}\label{le:g2onvs}
Let $V$ be a seven-dimensional vector space and $\varphi$ be a $\G_2$-structure on $V$. Then
$\varphi$ induces a unique Euclidean metric $g_{\varphi}$ and a unique metric volume form $vol_{\varphi}$ on $V$ such that each adapted basis $(e_1,\ldots,e_7)$ for $\varphi$ is an oriented orthonormal basis of $V$. For all $v,w\in V$, the Euclidean metric $g_{\varphi}$ and the metric volume form $vol_{\varphi}$ are given by the formula
\begin{equation*}
g_{\varphi}(v,w) vol_{\varphi}=(v\hook \varphi)\wedge (w\hook \varphi) \wedge \varphi.
\end{equation*}
\end{lemma}
\begin{remark}
$\G_2$-structures may be understood through the division algebra $(\mathbb{O},\langle \cdot,\cdot \rangle)$ of the octonions.
Therefore, let $1\in \mathbb{O}$ be the unit element of $\mathbb{O}$ and let $\Im\, \mathbb{O}:=\spa{1}^{\perp}$ be the imaginary octonions. Then $\varphi\in \L^3 \Im\, \mathbb{O}^*$ given by $\varphi(u,v,w):=\langle u\cdot v,w\rangle$ for $u,v,w\in \Im\, \mathbb{O}$ is a $\G_2$-structure on the seven-dimensional vector space $\Im\, \mathbb{O}$. Moreover, $\varphi$ induces in the sense of Lemma \ref{le:g2onvs} exactly the Euclidean metric $\langle \cdot,\cdot \rangle$ on $\Im\, \mathbb{O}$. For more details and for the relation of our definition to other definitions in the literature, we refer the reader to \cite{F}.
\end{remark}
Lemma \ref{le:g2onvs} tells us that a $\G_2$-structure $\varphi\in \L^3 V$ naturally induces a Euclidean metric $g_{\varphi}$ and a volume form $\vol_{\varphi}$ on $V$. Thus, we can define a Hodge star operator $\star_{\varphi}: \L^{*} V^*\rightarrow \L^{*} V^*$ by the usual requirement that for a $k$-form $\phi\in \L^k V^*$ the $(n-k)$-form $\star_{\varphi} \phi\in \L^{n-k} V^*$ is the unique $(n-k)$-form on $V$ such that for all $\psi\in \L^k V^*$ the identity
\begin{equation*}
\phi\wedge \psi= g_{\varphi} (\star_{\varphi} \phi,\psi) \vol_{\varphi}
\end{equation*}
holds. A short computation shows that the Hodge dual $\star_{\varphi} \varphi$ of the $\G_2$-structure $\varphi$ is given by
\begin{equation}\label{eq:HodgeDualG2}
\star_{\varphi} \varphi=e^{1234}+e^{1256}+e^{3456}-e^{2467}+e^{2357}+e^{1457}+e^{1367},
\end{equation}
where $e^1,\ldots,e^7$ is a dual basis of an adapted basis $(e_1,\ldots,e_7)$ for $\varphi$. Conversely, a four-form $\Psi\in \L^4 V^*$ of this kind has stabilizer $\G_2$ in $\GL^+(V)$. So if we fix an orientation on $V$, such a four-form gives rise to a Euclidean metric $g_{\Psi}$ and a $\G_2$-structure $\varphi$. In this case, $g_{\varphi}=g_{\Psi}$, $\star_{\varphi} \varphi=\Psi$, $\star_{\varphi} \Psi=\varphi$ and the orientation induced by $\varphi$ is the one fixed before \cite{Hi}. Hence, alternatively, it would also be possible to call such a four-form $\Psi$ together with an orientation a $\G_2$-structure. Even though this alternative definition is more appropriate in our case, we follow the convention in the literature and only call the three-form $\varphi$ a $\G_2$-structure. 

The set of all Hodge duals $\star_{\varphi} \varphi$ forms again an open orbit under $\GL(V)$ \cite{Hi}.
So for each Hodge dual $\star_{\varphi}\varphi$ there exists a small ball of radius $\epsilon_{\varphi}$
in $(\L^4 V^*,g_{\varphi})$ such that each four-form in this ball is again the Hodge dual of a $\G_2$-structure.
In fact, the sizes of these balls do not depend on the $\G_2$-structure $\varphi$ and the orbit is in this sense ``uniformly'' open. Namely, for two different $\G_2$-structures $\varphi_1,\varphi_2\in \L^3 V^*$ on $V$ the endomorphism of $V$ which maps an adapted basis of $\varphi_1$ onto an adapted basis of $\varphi_2$ induces an isometric isomorphism between $(\L^4 V^*,g_{\varphi_1})$ and $(\L^4 V^*,g_{\varphi_2})$. Hence, if a ball of radius $\epsilon$ with respect to $g_{\varphi_1}$ around $\star_{\varphi_1} \varphi_1$ lies in the orbit of all Hodge duals of $\G_2$-structures, then also a ball of radius $\epsilon$ with respect to $g_{\varphi_2}$ around $\star_{\varphi_2} \varphi_2$ lies in the orbit of all Hodge duals of $\G_2$-structures.

\begin{lemma}\label{le:universalconst}
There exists a universal constant $\epsilon_0>0$ such that if $\varphi\in \L^3 V^*$ is a $\G_2$-structure on a
seven-dimensional vector space $V$ and $\Psi\in \L^4 V^*$ is a four-form on $V$ which fulfills
\begin{equation*}
\left\| \Psi-\star_{\varphi}\varphi\right\|_{\varphi}<\epsilon_0
\end{equation*}
for the norm $\left\|\cdot\right\|_{\varphi}$ induced by the Euclidean metric $g_{\varphi}$ on $V$, then
$\Psi$ is the Hodge dual of a $\G_2$-structure on $V$.
\end{lemma}
\subsection{Cocalibrated $\G_2$-structures on manifolds and Lie algebras}\label{G2onmfs}
A \emph{$\G_2$-structure on a seven-dimensional manifold $M$} is by definition a reduction of the frame bundle
$\GL(M)$ to $\G_2\subseteq \GL_7(\bR)$. Since $\G_2$ is conjugated to the stabilizer of a $\G_2$-structure on the vector space $\bR^7$, 
there exists a one-to-one correspondence between $\G_2$-structures on $M$ and three-forms $\varphi\in \Omega^3 M$
such that $\varphi_p\in \L^3 T_p M^*$ is a $\G_2$-structure on $T_p M$ for all $p\in M$. In the following, we also call such a three-form $\varphi\in \Omega^3 M$ a \emph{$\G_2$-structure}. One can show that $\G_2$-structures exist exactly when $M$ is orientable and spin \cite{J}.

A $\G_2$-structure $\varphi\in \Omega^3 M$ on a seven-dimensional manifold induces a Riemannian metric and an orientation on $M$ by applying pointwise the construction described above. Hence, we get a Hodge dual operator $\star_{\varphi}:\Omega^* M\rightarrow \Omega^* M$ depending on the $\G_2$-structure $\varphi\in \Omega^3 M$. The $\G_2$-structure $\varphi\in \Omega^3 M$ is called \emph{cocalibrated} if $d \star_{\varphi} \varphi=0$. Note that a $\G_2$-structure is torsion-free if and only if $d\varphi=d\star_{\varphi}\varphi=0$, cf. \cite{FG}.

We concentrate on left-invariant $\G_2$-structures on Lie groups $\G$. These are in one-to-one correspondence to $\G_2$-structures on the corresponding Lie algebra $\g$. If we use this identification as usual to define a differential $d_{\g}$ on $\L^{*}\g^*$, we are able to speak of \emph{cocalibrated $\G_2$-structures on the Lie algebra $\g$}.

\subsection{Three-dimensional Lie algebras}\label{subsec:3dLA}
The classification of three-dimensional Lie algebras is well-known \cite{Bi} and given in the appendix in Table \ref{table:3d}. We highlight some aspects of the classification which we use later on in this article.
\begin{lemma}\label{le:Unimodular3d}
Let $\g$ be a three-dimensional unimodular Lie algebra.
\begin{enumerate}
\item
There exists a basis $e_1,e_2,e_3$ of $\g$ and $\tau_1,\tau_2,\tau_3\in \left\{-\frac{1}{2},0,\frac{1}{2}\right\}$ such that\\
$de^i=\tau_i \sum_{j,k=1}^3 \epsilon_{ijk} e^{jk}$ for $i=1,2,3$.
\item
$d( \mathfrak{g}^*)\wedge \left.\ker d\right|_{ \mathfrak{g}^*}=\{0\}$.
\item
There exists a linear map $g:\L^2 \g^*\rightarrow \left. \ker d\right|_{ \mathfrak{g}^*}$ such that for the map $G:\L^2 \g^*\rightarrow \L^3 \g^*$, $G(\omega):=\omega\wedge g(\omega)$ for $\omega\in \L^2 \g^*$, the identity $G^{-1}(0)=d( \mathfrak{g}^*)$ is true.
\item
If $\tau_i \tau_j\geq 0$ for all $i,j\in \{1,2,3\}$, i.e. $\mathfrak{g}\notin\{e(1,1),\mathfrak{so}(2,1)\}$, then $F^{-1}(0)=\ker d|_{ \g^*}$, where $F:\g^*\rightarrow \L^3 \g^*$ is defined by $F(\alpha):=d(\alpha)\wedge \alpha$ for $\alpha\in \g^*$.
\end{enumerate}
\end{lemma}
\begin{proof}
We use the well-known part (a) \cite{Bi} to show (b)-(d).
\begin{itemize}
\item[(b)]
Let $\omega=d \alpha$, $\alpha=\sum_{i=1}^3 a_i e^i\in  \g^*$ and $\beta=\sum_{i=1}^3 b_i e^i\in  \g^*$. Then
\begin{equation}\label{eq1:3dLA}
\omega=\sum_{i,j,k=1}^3 \tau_i a_i \epsilon_{ijk} e^{jk}
\end{equation}
and so
\begin{equation}\label{eq2:3dLA}
\begin{split}
\omega\wedge \beta &=\sum_{i,j,k,l=1}^3 \tau_i a_i b_l \epsilon_{ijk} e^{jkl}=\sum_{i,j,k,l=1}^3 \tau_i a_i b_l \epsilon_{ijk} \epsilon_{jkl} e^{123}\\
&=\left(\sum_{i=1}^3 2\tau_i a_i b_i\right) e^{123}.
\end{split}
\end{equation}
If $d\beta=\sum_{i,j,k=1}^3 \tau_i b_i \epsilon_{ijk} e^{jk}=0$, then $\tau_i b_i=0$ for all $i=1,2,3$ and so $\omega\wedge \beta=0$. This shows (b).
\item[(c)]
Let $\omega\in \L^2 \g^*$. Then $\omega=\sum_{i,j,k=1}^3 a_i \epsilon_{ijk} e^{jk}$ for unique $a_1,a_2,a_3\in \bR$. Set $g(\omega):=\sum_{i=1, \tau_i=0}^3 a_i e^i$. Then Equation (\ref{eq1:3dLA}) shows that $g(\omega)\in \ker d|_{ \g^*}$. Moreover,
\begin{equation*}
\begin{split}
\omega\wedge g(\omega)&=\sum_{i,j,k,l=1, \tau_l= 0}^3 a_i a_l \epsilon_{ijk} e^{jkl}=\left(\sum_{i,j,k,l=1, \tau_l= 0}^3 a_i a_l \epsilon_{jki} \epsilon_{jkl}\right) e^{123}\\
&=\left(\sum_{i,l=1, \tau_l= 0}^3 2 a_i a_l \delta_{il}\right)e^{123}=\left(\sum_{l=1, \tau_l= 0}^3 2 a_l^2\right)e^{123}=0
\end{split}
\end{equation*}
if and only if $\tau_l=0$ implies $a_l=0$ for $l=1,2,3$. But Equation (\ref{eq1:3dLA}) shows that this is equivalent to $\omega\in d(\g^*)$.
\item[(d)]
The signs of the non-zero $\tau_i$ are all the same due to the assertion. Let $\alpha=\sum_{i=1}^3 a_i e^i\in \g^*$, $a_1,a_2,a_3\in \bR$. Then Equation (\ref{eq2:3dLA}) implies that $d\alpha\wedge\alpha=0$ if and only if $\sum_{i=1}^3 \tau_i a_i^2=0$ and this is the case if and only if $\tau_i a_i=0$ for all $i=1,2,3$. But Equation (\ref{eq1:3dLA}) states that this is equivalent to $\alpha\in \ker d|_{ \g^*}$.
\end{itemize}
\end{proof}
Recall that a finite-dimensional Lie algebra is called \emph{almost Abelian} if it admits a codimension one Abelian ideal. All solvable three-dimensional Lie algebras $\g$ are almost Abelian. Namely, if $\g$ is additionally unimodular, then, by elementary Lie theory, there exists a codimension one ideal, which then has to be unimodular and so Abelian. If $\g$ is not unimodular, then the unimodular kernel gives a codimension one Abelian ideal. The differential of an almost Abelian Lie algebra is particularly simple, cf. \cite{F} or Lemma \ref{le:codimensiononeideal} below. Hence, we obtain
\begin{lemma}\label{le:3d}
Let $\g$ be a three-dimensional solvable Lie algebra. Then $\g^*$ admits a vector space decomposition $\g^*=W_2\op \spa{e^3}$, $W_2$ two-dimensional, and a linear map $f:W_2\rightarrow W_2$ such that $d\alpha=f(\alpha)\wedge e^3$ for all $\alpha\in W_2$ and $de^3=0$. If $\tr(f)\neq 0$, $\frac{\det(f)}{\tr(f)^2}$ only depends on the Lie algebra $\g$. Moreover, $\tr(f)=0$ exactly when $\g$ is unimodular.
\end{lemma}
\begin{remark}
The only non-solvable three-dimensional Lie algebras are $\mathfrak{so}(3)$ and $\mathfrak{so}(2,1)$.
\end{remark}
We recapitulate the definition of a contact form on an odd-dimensional Lie algebra.
\begin{definition}
Let $\g$ be a $(2m+1)$-dimensional Lie algebra. A \emph{contact form} on $\g$ is a one-form $\alpha\in  \g^*$ such that $\alpha\wedge (d\alpha)^m\neq 0$. For $m=1$, the case we are interested in, the condition simply is $\alpha\wedge d\alpha\neq 0$.
\end{definition}
In Section \ref{sec:class}, we need a classification of the three-dimensional Lie algebras which do admit a contact form. This classification is well-known \cite{D} and straightforward to prove:
\begin{lemma}\label{le:contact}
A three-dimensional Lie algebra does not admit a contact form if and only if $\g$ is solvable and $f$ as in Lemma \ref{le:3d} is a multiple of the identity. So $\g$ admits a contact-form if and only if $\g\notin\{\bR^3, \mathfrak{r}_{3,1}\}$. 
\end{lemma}
\subsection{Four-dimensional Lie algebras}\label{subsec:4dLA}
A classification of all four-dimensional Lie algebra has first been achieved by Mubarakzyanov \cite{Mu}. We give a complete list in Table \ref{table:4d}. In \cite{ABDO}, it is proven that each four-dimensional solvable Lie algebra admits a codimension one unimodular ideal. Since the only simple Lie algebras up to dimension four are $\so(3)$ and $\so(2,1)$, it is an immediate consequence of Levi's decomposition theorem that the non-solvable four-dimensional Lie algebras are exactly $\so(3)\op \bR$ and $\so(2,1)\op \bR$. This shows the first part of
\begin{lemma}\label{le:4dcodim1}
Let $\g$ be a four-dimensional Lie algebra. Then $\g$ admits a codimension one unimodular ideal $\uf$. $\uf$ is unique if and only if $\g$ is not unimodular or $\dim{[\g,\g]}=3$. In these cases $\uf$ is the unimodular kernel or the commutator ideal $[\g,\g]$ of $\g$, respectively.
\end{lemma}
\begin{proof}
If $\g$ is not unimodular, then the unimodular kernel has codimension one and each unimodular ideal of $\g$ is an ideal of the unimodular kernel. Thus, a codimension one unimodular ideal has to coincide with the unimodular kernel. The commutator ideal $[\g,\g]$ is a unimodular ideal and contained in each codimension one ideal. Thus, the uniqueness statement follows if $\dim{[\g,\g]}=3$.

If $\g$ is unimodular and $\dim{[\g,\g]}<3$, then, by inspecting Table \ref{table:4d}, we see that $\g=\h\op \bR$ with a three-dimensional unimodular solvable Lie algebra $\h$ or $\g=A_{4,1}$. In the former cases, the first summand $\h$ in $\h\op \bR$ is a unimodular codimension one ideal and the direct sum of an Abelian codimension one ideal of $\h$ and the $\bR$ summand gives a different unimodular codimension one ideal in $\g$. For $\g=A_{4,1}$, in the dual basis $e_1,e_2,e_3,e_4$ of the basis $e^1,e^2,e^3,e^4$ of $A_{4,1}^*$ given in Table \ref{table:4d}, the subspace $\spa{e_1,e_2,e_3}$ is an Abelian codimension one ideal whereas $\spa{e_1,e_2,e_4}$ is a codimension one ideal isomorphic to $\mathfrak{h}_3$.
\end{proof}
The exterior derivative of a Lie algebra with a codimension one unimodular ideal takes a particular nice form which turns out to be useful for many of the explicit computations we do in Section \ref{sec:class}.
\begin{lemma}\label{le:codimensiononeideal}
Let $\g$ be an $n$-dimensional Lie algebra which admits a codimension one unimodular ideal $\uf\subseteq \g$. Let $e_n\in \g\backslash \uf$ and $e^n\in \Ann{\uf}$ with $e^n(e_n)=1$. As described in Conventions \ref{conventions}, we identify $\L^*\Ann{e_n}$ with $\L^*\uf^*$ via the decomposition $\g=\uf\oplus \spa{e_n}$. Then the following statements are true:
\begin{enumerate}
\item
$d_{\g} e^n=0$ and there exists $f\in \mathfrak{gl}\left(\uf^*\right)$ such that
$d_{\g}\alpha=d_{\uf} \alpha+f(\alpha)\wedge e^n$ for all $\alpha\in \uf^*$.
\item
$d_{\g}(\omega\wedge e^n)=d_{\uf}(\omega)\wedge e^n$ for all $\omega\in \L^*\uf^*$.
\item
$d_{\g}(\Lambda^{n-2} \uf^*)\subseteq \Lambda^{n-2}\uf^*\wedge e^n$.
\item
$d_{\g}(\Lambda^{n-2}\uf^*\wedge e^n)=\{0\}$. Moreover, $d_{\g}(\Lambda^{n-1}\uf^*)=\{0\}$ exactly when $\g$ is unimodular.
\end{enumerate}
\end{lemma}
\begin{proof}
\begin{enumerate}
\item
For arbitrary $X,Y\in \g$, the commutator $[X,Y]$ is in $\uf$. Hence
$d_{\g}e^n(X,Y)=-e^n([X,Y])=0$ and so $d_{\g}e^n=0$. It is clear that there are linear maps $f:\uf^*\rightarrow \uf^*$ and $g:\uf^*\rightarrow \L^2\uf^*$ such that $d_{\g}(\alpha)=g(\alpha)+f(\alpha)\wedge e^n$ for all $\alpha\in \uf^*$. For $Z,W\in \uf$ we have $[Z,W]\in \uf$ and
\begin{equation*}
g(\alpha)(Z,W)=(d_{\g}\alpha)(Z,W)=-\alpha([Z,W])=(d_{\uf}\alpha)(Z,W).
\end{equation*}
Hence, $g(\alpha)=d_{\uf}(\alpha)$.
\item
Part (a) implies that $d_{\g}\omega=d_{\uf}\omega+f.\omega\wedge e^n$ for all $\omega\in \Lambda^{k}\uf^*$, where $(f,\omega)\mapsto f.\omega$ is the natural action of $f\in \mathfrak{gl}\left(\uf^*\right)$ on $\omega\in \L^k \uf^*$. Then (a) implies $d_{\g}(\omega\wedge e^n)=d_{\g}(\omega)\wedge e^n=d_{\uf}(\omega)\wedge e^n$ as claimed.
\item
We have $d_{\g}\omega=d_{\uf}\omega+f.\omega\wedge e^n$ for all $\omega\in\Lambda^{n-2}\uf^*$. But $\uf$ is unimodular, which is equivalent to the fact that all $(n-2)$-forms on $\uf$ are $d_{\uf}$-closed. Hence, $d_{\g}\omega=f.\omega \wedge e^n\in \Lambda^{n-2}\uf^*\wedge e^n$ as claimed.
\item
Part (a) and (c) directly imply $d_{\g}(\Lambda^{n-2}\uf^*\wedge e^n)=\{0\}$. Since $\g$ is unimodular exactly when all $(n-1)$-forms are $d_{\g}$-closed, the first part implies that $d_{\g}(\Lambda^{n-1}\uf^*)=\{0\}$ exactly when $\g$ is unimodular.
\end{enumerate}
\end{proof}
We recapitulate the definition of a symplectic two-form on an even-dimensional Lie algebra:
\begin{definition}
Let $\g$ be a Lie algebra of dimension $2m$. A closed two-form $\omega\in \L^2 \g^*$ is called \emph{symplectic} if it is \emph{non-degenerate}, i.e. $\omega^m\neq 0$. For the case we are interested in, namely $m=2$, this simply means $\omega^2\neq 0$.
\end{definition}
All symplectic four-dimensional Lie algebras have been identified and also all symplectic two-forms (up to isomorphisms) have been determined by Ovando in \cite{O}. We give a new proof of some part of the results in order to relate the existence of one or more symplectic two-forms satisfying certain compatibility relations to the dimensions of the cohomology groups of $\g$ and of a codimension one unimodular ideal $\uf$.
\begin{lemma}\label{le:symplect}
Let $\g$ be a four-dimensional Lie algebra and assume that $\g$ is almost Abelian with codimension one Abelian ideal $\uf$ or $\g$ is not unimodular and the unimodular kernel $\uf$ is not isomorphic to $e(1,1)$. Then $\g$ admits a
\begin{equation*}
D:=h^2(\mathfrak{g})-h^1(\mathfrak{g})-h^1(\mathfrak{u})+4
\end{equation*}
-dimensional subspace of $\L^2 \g^*$ in which each non-zero element is symplectic.
\end{lemma}
\begin{remark}\label{re:4dLAs}
\begin{itemize}
\item
Note that we do not claim in Lemma \ref{le:symplect} that $D=h^2(\mathfrak{g})-h^1(\mathfrak{g})-h^1(\mathfrak{u})+4$ is the maximal dimension of a subspace as in the statement. However, our main result Theorem \ref{th:main} together with Proposition \ref{pro:existence} below imply that $D$ is, in fact, the maximal dimension.
\item
Lemma \ref{le:symplect} applies to all but five Lie algebras: The only non-unimodular four-dimensional Lie algebra with unimodular kernel $\uf$ isomorphic to $e(1,1)$ is $\mathfrak{r}_2\op \mathfrak{r}_2$. In the basis given in Table \ref{table:4d}, the two-form $e^{14}+e^{23}$ is symplectic. One can show that the maximal dimension of a subspace $V\subseteq \L^2 (\mathfrak{r}_2\op \mathfrak{r}_2)^*$, in which each non-zero element is symplectic, is one. The unimodular four-dimensional Lie algebras which are not almost Abelian are the two non-solvable ones $\so(3)\op \bR$ and $\so(2,1)\op \bR$ and two other Lie algebras, namely $A_{4,8}$ and $A_{4,10}$. All four do not admit any symplectic two-form.
\end{itemize}
\end{remark}
\begin{proof}[Proof of Lemma \ref{le:symplect}]
Fix a norm $\left\|\cdot \right\|$ on $\g^*\oplus \L^2 \g^*$ and identify $\L^4 \g^*\cong \bR$ for the rest of the proof. Choose an element $e_4\in \g\backslash \uf$ and let $e^4\in \Ann{\uf}$ be such that $e^4(e_4)=1$. As usual, we identify $\Ann{e_4}\cong \uf^*$ via the decomposition $\g=\uf\oplus \spa{e_7}$. By Lemma \ref{le:codimensiononeideal}, there exists $f\in\mathfrak{gl}\left(\uf^*\right)$ such that $d_{\g} \beta=d_{\uf} \beta+f(\beta)\wedge e^4$ for all $\beta\in  \uf^*$. We fix a complement $V$ of $\ker d_{\uf}|_{ \uf^*}$ in $ \uf^*$ and set
\begin{equation*}
W_{\lambda}:=\left\{\left.\omega+\lambda g(\omega)\wedge e^4\right|\omega\in \ker d_{\g}|_{\L^2 \uf^*} \right\}\subseteq \ker d_{\g}|_{\L^2 \uf^*}\oplus \ker d_{\uf}|_{\uf^*}\wedge e^4
\end{equation*}
for $\lambda\neq 0$ with $g:\L^2 \uf^*\rightarrow \ker d_{\uf}|_{\uf^*}$ as in Lemma \ref{le:Unimodular3d} (c), i.e. $g(\omega)\wedge \omega=0$ if and only if $\omega\in d_{\uf}(\uf^*)$. We claim that there is $\lambda\neq 0$ such that $U:=d_{\g}(V)+W_{\lambda}$ consists, with the exception of the origin, of symplectic two-forms and that the dimension of $U$ is equal to $D=h^2(\mathfrak{g})-h^1(\mathfrak{g})-h^1(\mathfrak{u})+4$. Note that the closure of all elements in $U$ is clear. We divide the proof into six steps.

\emph{Step I: All non-zero elements in $d_{\g}(V)$ are symplectic and $d_{\g}|_V:V\rightarrow d_{\g}(V)$ is an isomorphism:}

If $V=\{0\}$, then there is nothing to show. Otherwise our assumptions imply that $\g$ is not unimodular and so $d_{\g}(\L^3 \uf^*)\neq \{0\}$ by Lemma \ref{le:codimensiononeideal}. Let $\alpha\in V\backslash \{0\}$. By definition of $V$, $d_{\uf} \alpha\neq 0$ and so Lemma \ref{le:Unimodular3d} (d) tells us that $\L^3 \uf^* \ni d_{\uf} \alpha\wedge \alpha\neq 0$. Hence
$d_{\g}(d_{\uf} \alpha\wedge \alpha)\neq 0$ and so
\begin{equation*}
d_{\g} \alpha\wedge d_{\g} \alpha=d_{\g}(\alpha\wedge d_{\g} \alpha)=d_{\g}(\alpha\wedge d_{\uf} \alpha+ \alpha\wedge f(\alpha) \wedge e^4)=d_{\g}(\alpha\wedge d_{\uf} \alpha)\neq 0.
\end{equation*}
So $d_{\g} \alpha$ is non-degenerate and, in particular, $d_{\g} \alpha\neq 0$. This proves Step I.

\emph{Step II: $f(V)$ is a complement of $\ker d_{\uf}|_{\uf^*}$ in $ \uf^*$ and $d_{\g}(V)\cap W_{\lambda}=$\\
$d_{\g}(V)\cap \left(\ker d_{\g}|_{\L^2 \uf^*}\oplus \ker d_{\uf}|_{\uf^*}\wedge e^4\right)=\{0\}$ for all $\lambda\neq 0$:}

The inequality $0\neq d_{\g} \alpha\wedge d_{\g} \alpha=2 d_{\uf} \alpha\wedge f(\alpha)\wedge e^4$
for $\alpha \in V \backslash \{0\}$ implies that $f|_V$ is injective and so $\dim{V}=\dim{f(V)}$. By Lemma \ref{le:Unimodular3d} (b), $\ker d_{\uf}|_{ \uf^*}\wedge d_{\uf}( \uf^*)=\{0\}$. Thus, $f(V)$ is a complement of $\ker d_{\uf}|_{ \uf^*}$ in $ \uf^*$. Let $\omega\in d_{\g}(V)\cap \left(\ker d_{\g}|_{\L^2 \uf^*}\oplus \ker d_{\uf}|_{\uf^*}\wedge e^4\right)$. Then there are $\alpha\in V,$ $\omega_1\in \ker d_{\g}|_{\Lambda^2 \uf^*}$ and $\beta\in \ker d_{\uf}|_{\uf^*}$ such that
\begin{equation*}
\omega=d_{\uf}\alpha +f(\alpha)\wedge e^4=\omega_1+\beta \wedge e^4.
\end{equation*}
This implies $f(\alpha)=\beta\in \ker d_{\uf}|_{ \uf^*}$ and so, since $f(V)$ is a complement of $\ker d_{\uf}|_{ \uf^*}$ in $\uf^*$, $\beta=0$. Now $f|_V$ is injective and so we must have $\alpha=0$, which ultimately implies $\omega=0$. This finishes the proof of Step II.

\emph{ Step III: $\dim{d_{\g}(V)\op W_{\lambda}}=h^2(\g)-h^1(\g)-h^1(\uf)+4$}:

Note that the dimension of $W_{\lambda}$ is equal to the dimension of $\ker d_{\g}|_{\L^2 \uf^*}$ and that the dimension of $\ker d_{\g}|_{\L^2 \g^*}$ is $h^2(\g)+4-h^1(\g)$. Therefore it suffices to show
\begin{equation*}
\ker d_{\g}|_{\L^2 \g^*}=\ker d_{\g}|_{\L^2 \uf^*}\op \ker d_{\uf}|_{ \uf^*}\wedge e^4\op d_{\g}(V)
\end{equation*}
to get the statement about the dimension of $d_{\g}(V)\op W_{\lambda}$. The inclusion ``$\supseteq$'' is obvious. For the other inclusion, let
$\omega\in \ker d_{\g}|_{\L^2 \g^*}$. Then there exists $\omega_1\in \L^2 \uf^*$ and $\beta\in  \uf^*$ such that $\omega=\omega_1+\beta\wedge e^4$. Since $f(V)$ is a complement of $\ker d_{\uf}|_{ \uf^*}$ in $ \uf^*$, there exists $\alpha\in V$ with $\beta-f(\alpha)\in \ker d_{\uf}|_{ \uf^*}$. Then
\begin{equation*}
\omega-(\beta-f(\alpha))\wedge e^4-d_{\g} \alpha=\omega_1+\beta\wedge e^4-(\beta-f(\alpha))\wedge e^4 -d_{\uf} \alpha-f(\alpha)\wedge e^4=\omega_1-d_{\uf}\alpha \in \L^2\uf^*
\end{equation*}
and $\omega-(\beta-f(\alpha))\wedge e^4-d_{\g} \alpha$ is $d_{\g}$-closed. Hence, $\omega\in \ker d_{\g}|_{\L^2 \uf^*}\op \ker d_{\uf}|_{ \uf^*}\wedge e^4\op d_{\g}(V)$.

\emph{Step IV: $\ker d_{\g}|_{\L^2 \uf^*} \cap d_{\uf} ( \uf^*)=\{0\}$}:

Let $\omega\in \ker d_{\g}|_{\L^2 \uf^*} \cap d_{\uf} ( \uf^*)$. Then $\omega=d_{\uf}\beta$ for some $\beta\in \uf^*$ and $d_{\g} \omega=0$. We may assume that $\beta\in V$. But then
\begin{equation*}
0=d_{\g} \omega= d_{\g}(d_{\g}\beta-f(\beta)\wedge e^4)=-d_{\uf}( f(\beta))\wedge e^4.
\end{equation*}
Since $f(V)$ is a complement of $\ker d_{\uf}|_{ \uf^*}$ in $ \uf^*$ and $f|_{V}$ is injective we get $\beta=0$ and so $\omega=0$ as claimed.

\emph{Step V: Norm estimates:}

Note first that the identity
\begin{equation*}
(d_{\g} \alpha)^2=2 d_{\uf} \alpha\wedge f(\alpha)\wedge e^4
\end{equation*}
and the fact that $f|_V$ and $d_{\uf}|_{V}$ are injective imply the existence of a constant $A>0$ such that
\begin{equation}\label{eq:constantA}
|(d_{\g} \alpha)^2|\geq A \left\|\alpha\right\|^2.
\end{equation}
Note further the sign of $(d_{\g}\alpha)^2\in \Lambda^4 \g^*\cong \bR$ for $\alpha\neq 0$ does not depend on $\alpha$. Namely, let $F:V\rightarrow \bR, \, F(\alpha):=(d_{\g}\alpha)^2$. For $\dim{V}>1$ the set $V\backslash\{0\}$ is connected, while $F(V\backslash \{0\})$ is disconnected if the sign depends on $\alpha\neq 0$, contradicting the continuity of $F$. If $\dim{V}=1$ then the statement follows from the fact that $F$ is homogeneous of degree two in $\alpha$.

Next we consider the space $W_{\lambda}$ for arbitrary $\lambda\neq 0$. Lemma \ref{le:Unimodular3d} (c) tells us that 
\begin{equation*}
(\omega+\lambda g(\omega)\wedge e^4)^2=2\lambda\, \omega\wedge g(\omega)\wedge e^4=0
\end{equation*}
for $\omega\in \ker d_{\g}|_{\L^2 \uf^*}$ implies $\omega\in d_{\uf}(\uf^*)$. But Step IV tells us that then $\omega=0$. Thus, there exists $C>0$, independent of $\lambda$, such that
\begin{equation}\label{eq:constantC}
|(\omega+\lambda g(\omega)\wedge e^4)^2|\geq C |\lambda| \left\|\omega\right\|^2
\end{equation}
for all $\omega\in \ker d_{\g}|_{\L^2 \uf^*}$. Note that for fixed $\lambda\neq 0$, arguing as above, we see that the sign of $(\omega+\lambda g(\omega)\wedge e^4)^2\in \bR$ does not depend on $\omega$. But it gets reversed if we reverse the sign of $\lambda$. Hence, we may assume that it is chosen such that $\omega_1^2\cdot \omega_2^2>0$ for all $\omega_1\in d_{\g}(V)\backslash \{0\}$, $\omega_2\in W_{\lambda}\backslash \{0\}$. By Lemma \ref{le:Unimodular3d} (b), the identity $d_{\uf}\alpha\wedge g(\omega)=0$ is true for all $\alpha\in V$ and $\omega\in \ker d_{\g}|_{\L^2 \uf^*}$. Thus,
\begin{equation*}
2 d_{\g}\alpha\wedge (\omega+\lambda g(\omega)\wedge e^4)=2 (d_{\uf}\alpha+f(\alpha)\wedge e^4)\wedge (\omega+\lambda g(\omega)\wedge e^4)=2 f(\alpha)\wedge e^4\wedge \omega
\end{equation*}
and there exists a constant $B>0$ such that
\begin{equation}\label{eq:constantB}
|2 d_{\g}\alpha\wedge (\omega+\lambda g(\omega)\wedge e^4)|\leq B \left\|\alpha \right\| \left\|\omega\right\|.
\end{equation}
\emph{Step VI: All non-zero elements in $d_{\g}(V)\oplus W_{\lambda}$ are symplectic for appropriate $\lambda\neq 0$:}

Let $0\neq \omega_0=\omega_1+\omega_2\in d_{\g}(V)\oplus W_{\lambda}$ with $\omega_1=d_{\g} \alpha\in d_{\g}(V)$ for some $\alpha\in V$ and $\omega_2=\omega+\lambda g(\omega)\wedge e^4\in W_{\lambda}$ for some $\omega\in \ker d_{\g}|_{\L^2 \uf^*}$. By the previous steps, we only have to consider the case when $\omega_1\neq 0$ and $\omega_2\neq 0$. Then both $\alpha$ and $\omega$ are not zero by the Equations (\ref{eq:constantA}) and (\ref{eq:constantC}). The discriminant of the polynomial $\omega_0^2=(\omega_1+X\omega_2)^2=\omega_2^2+2X\,\omega_1\wedge \omega_2+X^2\,\omega_1^2$ is given by
\begin{equation*}
(2\omega_1\wedge \omega_2)^2-4\,\omega_1^2\cdot \omega_2^2\leq B^2 \left\|\alpha \right\|^2 \left\|\omega\right\|^2-4 |\lambda| A\, C \left\|\alpha \right\|^2 \left\|\omega\right\|^2=(B^2-4 |\lambda| A\, C)\left\|\alpha \right\|^2 \left\|\omega\right\|^2,
\end{equation*}
where we used Equations (\ref{eq:constantA}), (\ref{eq:constantC}) and (\ref{eq:constantB}) and the fact that the sign of $\omega_1^2\cdot \omega_2^2$ may be assumed to be positive. But for sufficiently large $|\lambda|$, independent of $\alpha$ and $\omega$, this is negative and the quadratic polynomial in $X$ does not have a real root. In particular, $X=1$ is not a real root and so $\omega_0=\omega_1+\omega_2$ is non-degenerate. This finishes the proof.

\end{proof}
In Lemma \ref{le:codimensiononeideal}, we gave a description of the exterior derivative of $n$-dimensional Lie algebras having a codimension one unimodular ideal $\uf$. If $n=4$ and $\uf=\h_3$, the next lemma shows that we can do better. For a proof, we refer the reader to \cite{ABDO}.
\begin{lemma}\label{le:4dbasis}
If $\g$ is a four-dimensional Lie algebra $\g$ which possesses an ideal $\uf$ isomorphic to $\h_3$, then there exist an element $e_4\in \g\backslash \uf$, an element $e^1\in \uf^*\cong \Ann{e_4}$, a two-dimensional subspace $V_2\subseteq \uf^*$ with $\spa{e^1}\op V_2=\uf^*$, a linear map $F:V_2\rightarrow V_2$ and a non-zero two-form $\nu\in \L^2 V_2\backslash \{0\}$ such that $de^1=\tr(F) e^{14}+\nu$, $d \alpha=F(\alpha)\wedge e^4$ for all $\alpha\in V_2$ and $de^4=0$. Here, $e^4$ is the element in $\Ann{\uf}$ with $e^4(e_4)=1$. In this case, $\tr(F)=0$ if and only if $\g$ is unimodular.
\end{lemma}
\subsection{Algebraic invariants}\label{subsec:alginv}
Westwick introduced certain kinds of algebraic invariants to classifiy the orbits of three-forms on a seven-dimensional vector space $V$ under $\GL(V)$ \cite{W}. In \cite{F}, we already used these invariants to get obstructions to the existence of $\G_2$-structures. For that reason we determined the values of these invariants for the orbit of all Hodge duals of $\G_2$-structures in $\L^4 V^*$. Here, we briefly recapitulate the definitions and results.
\begin{definition}\label{def:alginvariants}
Let $V$ be an $n$-dimensional vector space. The \emph{Grassman cone} $G_k(V)$ consists of all \emph{decomposable} $k$-forms on $V$, i.e. of all those $k$-forms $\psi\in \L^k V^*$ such that there are $k$ one-forms $\alpha_1,\ldots,\alpha_k$ with $\psi=\alpha_1\wedge\ldots\wedge \alpha_k$. The \emph{length} $l(\phi)$ of an arbitrary $k$-form $\phi\in \L^k V^*$ is defined as the minimal number $m$ of decomposable $k$-forms $\phi_1,\ldots,\phi_{m}$
which is needed to write $\phi$ as the sum of $\phi_1,\ldots,\phi_m$, i.e. as $\phi=\sum_{i=1}^m \phi_i$. The \emph{rank} $\rk(\phi)$ of $\phi$ is the dimension of the subspace
\begin{equation*}
[\phi]:=\bigcap \left\{ \phi\in \L^k U| \textrm{$U$ is a subspace of $V^*$} \right\}
\end{equation*}
or, equivalently, the rank of the linear map $T:V\rightarrow \L^{k-1} V^*$, $T(v)=v\hook \phi$. $[\phi]$ is also called the \emph{support (of $\phi$)}. For a vector $v\notin \ker T$ and a subspace $W\subseteq V$ such that $W\op \spa{v}\op \ker T=V$ is a direct vector space sum, we set $\rho(v,W):=(v\hook \phi)|_W \in \L^{k-1} W^*$ and $\Omega(W):=\phi|_W\in \L^k W^*$. We introduce two more algebraic invariants by
\begin{equation*}
\begin{split}
r(\phi)&:=\min \left\{l(\Omega)|\Omega=\Omega(W)\in \L^k W^* \textrm{, $\dim{W}=(\rk(\phi)-1)$, $W\cap \ker T=\{0\}$} \right\}, \\
m(\phi)&:=\min \left\{ l(\rho)|\rho=\rho(v,W)\in \L^{k-1} W^* \textrm{, $v\notin \ker T$, $W\op \spa{v}\op \ker T=V$}\right\}.
\end{split}
\end{equation*}
\end{definition}
\begin{remark}
An equivalent description of the numbers $r(\phi)$ and $m(\phi)$ is obtained as follows:

Let $\phi\neq 0$, $\alpha\in [\phi]$, $\alpha\neq 0$ and $U$ be a complement of $\spa{\alpha}$ in $[\phi]$. Denote by $\rho(\alpha, U)\in \L^{k-1} U$ and $\Omega(\alpha,U)\in \L^k U$ the unique three- and four-form on $V$ such that
\begin{equation*}
\phi=\rho(\alpha,U)\wedge \alpha+\Omega(\alpha,U).
\end{equation*}
Then
\begin{equation*}
\begin{split}
r(\phi)&=\min\{l(\Omega)|\Omega=\Omega(\alpha,U)\in \L^k U,\, \alpha\in [\phi]\backslash \{0\},\, U\oplus \spa{\alpha}=[\phi]\}, \\
m(\phi)&=\min\{l(\rho)|\rho=\rho(\alpha,U)\in \L^{k-1} U,\, \alpha\in [\phi]\backslash \{0\},\, U\oplus \spa{\alpha}=[\phi] \}.
\end{split}
\end{equation*}
We will mostly work with this description.
\end{remark}
\begin{remark}\label{re:alginv}
\begin{itemize}
\item
The numbers $l(\phi),\rk(\phi),r(\phi)$ and $m(\phi)$ for a $k$-form $\phi\in \L^k V^*$ are invariant under isomorphisms $f^*:\L^k V^*\rightarrow \L^k W^*$ induced by isomorphisms $f:W\rightarrow V$. In particular, these four numbers are invariants of orbits under the natural action of $\GL(V)$ on $\L^k V^*$. Moreover, if $W:=V\op \spa{w}$, $w\neq 0$ and $\alpha\in  \Ann{V}$, $\alpha\neq 0$, then $l(\alpha \wedge \phi)=l(\phi)$.
\item
Let $\phi\in \L^k V^*$ be a $k$-form and set $T:V\rightarrow \Lambda^{k-1} V^*$, $T(w):=w\hook \phi$ as above. Let $v\notin \ker T$ and let $W_1$, $W_2$ be two subspaces of $V$ such that $V=\spa{v}\oplus W_i\oplus \ker T$ for $i=1,2$. Let $\rho(v,W_i):=(v\hook \phi)|_{W_i}$ for $i=1,2$ and denote by $\mathrm{pr}_{W_2}:V\rightarrow W_2$ the projection of $V$ onto $W_2$ along $\spa{v}\oplus \ker T$. Then $f:W_1\rightarrow W_2$, $f:=\mathrm{pr}_{W_2}|_{W_1}$ is an isomorphism with $f^* \rho(v,W_2)=\rho(v,W_1)$. In this sense, $\rho(v,W_i)$ essentially only depends on $v$ and the values of the above introduced algebraic invariants coincide for $\rho(v,W_1)$ and $\rho(v,W_2)$.
\item
A two-form $\omega\in \L^2 V^*$ has length $l$ if and only if $\omega^l\neq 0$ and $\omega^{l+1}$ is zero. Hence the maximal length of a two-form on an $n$-dimensional vector space is $\left\lfloor \frac{n}{2} \right\rfloor$. If the dimension $n$ of $V$ is even, i.e. $n=2m$, then the non-degenerate two-forms are exactly those of maximal length $m$.
\item
There exists an isomorphism $\delta:\L^k V^*\rightarrow \L^{n-k} V^*$ such that $l(\phi)=l(\delta(\phi))$ for all $\phi\in \L^k V^*$ \cite{F}. Moreover, if $V=V_1\op V_2$ as vector spaces then we may assume that $\delta:\L^{k_1} V_1^*\wedge \L^{k_2} V_2^*\rightarrow \L^{n_1-k_1} V_1^*\wedge \L^{n_2-k_2} V_2^*$, where $n_i:=\dim{V_i}$, $i=1,2$ (e.g. we may choose an appropriate Hodge star operator).
\end{itemize}
\end{remark}
The following lemma was proven in \cite{F}. 
\begin{lemma}\label{le:alginv1}
Let $\varphi$ be a $\G_2$-structure on a seven-dimensional vector space $V$. Let $v\in V\backslash\{0\}$ and $W$ be a complement of $\spa{v}$ in $V$. Then
\begin{enumerate}
\item
$(\rk(\star_{\varphi} \varphi),l(\star_{\varphi} \varphi),r(\star_{\varphi} \varphi),m(\star_{\varphi} \varphi))=(\rk(\varphi),l(\varphi),r(\varphi),m(\varphi))=(7,5,3,3).$
\item
The three-form $\rho:=(v\hook \star_{\varphi} \varphi)|_W\in \L^3 W^*$
fulfills
\begin{equation*}
(\rk( \rho ),l(\rho ),r(\rho ),m(\rho ) )= (6,3,2,2).
\end{equation*}
\item
The four-form $\Omega:=\star_{\varphi} \varphi|_W\in \L^4 W^*$ fulfills
\begin{equation*}
(\rk( \Omega ),l(\Omega),r(\Omega ),m(\Omega ) )= (6,3,1,2).
\end{equation*}
\end{enumerate}
\end{lemma}
\begin{remark}
We like to note that Lemma \ref{le:alginv1} may also be proved more directly. Therefore, note that by Remark \ref{re:alginv} we may assume that the decomposition $\g=\spa{v}\oplus W$ is orthogonal with respect to the induced metric. It is well-known, see e.g. \cite{CLSS}, that then $\Omega(v,W)=\frac{1}{2}\omega^2$ for some non-degenerate $\omega\in \Lambda^2 W^*$ such that $(\omega,\rho(v,W))\in \Lambda^2 W^*\times \Lambda^3 W^*$ is an $\SU(3)$-structure on $W$. \cite{W} gives us now the values of the algebraic invariants for $\rho(v,W)$ and the ones for $\Omega(v,W)=\frac{1}{2}\omega^2$ are easily computed.
\end{remark}
We end the subsection by proving the following technical lemma which we will apply in some of the proofs in Section \ref{sec:class}.
\begin{lemma}\label{le:alginv2}
Let $V$ be a six-dimensional vector space.
\begin{enumerate}
\item
Let $V=V_3\op W_3$ be a decomposition into two vector spaces of dimension three and let $\Omega=\Omega_1+\Omega_2\in \L^4 V^*$ with $\Omega_1\in \L^2 V_3^*\wedge \L^2 W_3^*$ and $\Omega_2 \in  V_3^* \wedge \L^3 W_3^*$ be a four-form of length three. Then the length of $\Omega_1$ is also three.
\item
Let $V=V_4\op V_2$ be a decomposition into a vector space $V_4$ of dimension four and a vector space $V_2$ of dimension two. Let $\rho$ be a three-form of rank six with $r(\rho)= 2$ such that $\rho\in \L^2 V_4^*\wedge  V_2^* \op  V_4^*\wedge \L^2 V_2^*$.  Then, for any basis $\alpha_1,\alpha_2$ of $V_2^*$, the unique two-forms $\omega_1,\omega_2\in \L^2 V_4^*$ such that $\rho-\sum_{i=1}^2 \omega_i\wedge \alpha_i \in  V_4^*\wedge \L^2 V_2^*$ span a two-dimensional subspace in $\L^2 V_4^*$ in which each non-zero element is of length two.
\end{enumerate}
\end{lemma}
\begin{proof}
\begin{enumerate}
\item
We use a dual isomorphism $\delta$ adapted to the splitting as explained above. Then $\delta(\Omega_1)\in  V_3^*\wedge  W_3^*$ and $\delta(\Omega_2)\in \L^2 V_3^*$. Since the length of $\delta(\Omega)$ is three, we have $0\neq \delta(\Omega)^3=(\delta(\Omega_1)+\delta(\Omega_2) )^3=\delta(\Omega_1)^3$. Thus, $\delta(\Omega_1)$ and so $\Omega_1$ has length three.
\item
There is $\beta\in  V_4^*$ such that $\rho=\omega_1\wedge \alpha_1+\omega_2\wedge \alpha_2+\beta\wedge \alpha_1\wedge \alpha_2$. We have to show that $l(a\omega_1+b\omega_2)=2$ for all $(a,b)\neq (0,0)$. Without loss of generality, we may assume $a\neq 0$ and then even $a=1$. If we rewrite $\rho$ as
\begin{equation*}
\rho= (\omega_2+\beta\wedge \alpha_1)\wedge (\alpha_2-b \alpha_1)+ (\omega_1+b\omega_2)\wedge \alpha_1
\end{equation*}
we see that $(\omega_1+b\omega_2)\wedge \alpha_1\in \L^3 (V_4^*\op \spa{\alpha_1})$ and $(\omega_2+\beta\wedge \alpha_1)\in \L^2 (V_4^*\op \spa{\alpha_1})$. Thus, $r(\rho)=2$ implies $l((\omega_1+b\omega_2)\wedge \alpha_1)\geq 2$ (consider $V^*=(V_4^*\op \spa{\alpha_1})\op \spa{\alpha_2-b\alpha_1}$) and so $l(\omega_1+b\omega_2)\geq 2$. Since the maximal length of a two-form in four dimensions is two, we get $l(\omega_1+b\omega_2)= 2$.
\end{enumerate}
\end{proof}
\subsection{Construction of $\G_2$-structures}\label{subsec:constructionG2}
In this subsection, we show how one may construct a $\G_2$-structure $\varphi\in \L^3 V^*$ on a seven-dimensional vector space $V$ from a given decomposition $V=V_4\oplus V_3$ of $V$ into a four-dimensional subspace $V_4$ and a three-dimensional subspace $V_3$ and from certain two-forms on $V_4$ and $V_3$. The decomposition $V=V_4\oplus V_3$ will then be an \emph{adapted splitting for $\varphi$}.
\begin{definition}\label{def:adaptedsplittingvspaces}
Let $\varphi\in \L^3 V^*$ be a $\G_2$-structure on a seven-dimensional vector space $V$. A splitting $V=V_4\op V_3$ is called \emph{adapted (for $\varphi$)} if there exists an adapted basis $(f_1,\ldots,f_7)$ for $\varphi$ such that $f_1,\ldots, f_4$ is a basis of $V_4$ and $f_5,f_6,f_7$ is a basis of $V_3$. If $M$ is a seven-dimensional manifold and $\varphi\in \Omega^3 M$ is a $\G_2$-structure on $M$, then an \emph{adapted splitting (for $\varphi$)} is a decomposition $TM=E_4\oplus E_3$ of $TM$ into subbundles $E_4$ and $E_3$ such that for all $p\in M$ the vector space decomposition $T_p M=(E_4)_p\oplus (E_3)_p$ is an adapted splitting for $\varphi_p\in \L^3 T_p M^*$.
\end{definition}
The following lemma follows directly from Equation (\ref{eq:HodgeDualG2}) and the fact that adapted bases are orthonormal bases.
\begin{lemma}\label{le:AdaptedSplitting}
Let $V$ be a seven-dimensional vector space, $\varphi\in \L^3 V^*$ be a $\G_2$-structure on $V$ and $V=V_4\oplus V_3$ be an adapted splitting. Then the decomposition $V=V_4\oplus V_3$ is orthogonal with respect to $g_{\varphi}$ and there exist a non-zero $\Omega_1\in \L^4 V_4^*$ and a non-zero $\Omega_2\in \L^2 V_4^*\wedge \L^2 V_3^*$ such that
\begin{equation}\label{eq:AdaptedSplitting1}
\star_{\varphi}\varphi=\Omega_1+\Omega_2.
\end{equation}
Moreover, if $\tilde{\varphi}\in \L^3 V$ is a $\G_2$-structure with adapted basis $(F_1,\ldots,F_7)$, $F_j=\frac{1}{\lambda} f_j$ for $j=1,2,3,4$, $F_l=f_l$ for $l=5,6,7$, then the splitting $V=V_4\oplus V_3$ is also adapted for $\tilde{\varphi}$, $\, g_{\tilde\varphi}|_{V_4}=\lambda^2 g_{\varphi}|_{V_4},g_{\tilde\varphi}|_{V_3}=g_{\varphi}|_{V_3}$ and
\begin{equation}\label{eq:AdaptedSplitting2}
\star_{\tilde{\varphi}}\tilde{\varphi}=\lambda^4 \Omega_1+\lambda^2 \Omega_2.
\end{equation}
\end{lemma}
\begin{remark}
An adapted splitting is also called \emph{coassociative/associative splitting}, see \cite{AS}. This is due to the fact that $V_3$ is a calibrated subspace for $\varphi$ and $V_4$ is a calibrated subspace for $\star_{\varphi}\varphi$. However, since we do not need calibrations at all in this article, we prefer the term ``adapted splitting''.
\end{remark}
Next, we give equivalent conditions when a subspace $W\subseteq \L^2 V^*$ of the two-forms on a four-dimensional vector space $V$ consists, with the exception of the origin, solely of two-forms of length two.
\begin{lemma}\label{le:subspaceslengthtwo}
Let $V$ be a four-dimensional vector space, $k\in \{0,1,2,3\}$, $\omega_1,\ldots,\omega_k\in \L^2 V^*$ be arbitrary two-forms on $V$, $\tau\in \L^4 V^*\backslash\{0\}$ and $\pi$ be an arbitrary permutation of $\{1,2,3\}$. Set $W:=\spa{\omega_1,\ldots,\omega_k}$,  $\tilde{\omega}_1:=e^{12}+e^{34}\in \L^2 \left(\bR^4\right)^*$, $\tilde\omega_2:=e^{13}-e^{24}\in \L^2 \left(\bR^4\right)^*$, $\tilde{\omega}_3:=e^{14}+e^{23}\in \L^2 \left(\bR^4\right)^*$. Moreover, define the symmetric matrix $H=(h_{ij})_{ij}\in \bR^{k\times k}$ by $\omega_i\wedge \omega_j=h_{ij}\tau$ for $i,j=1,\ldots, k$. Then the following are equivalent:
\begin{enumerate}
\item[(i)]
$W$ is $k$-dimensional and each element in $W\backslash \{0\}$ has length two.
\item[(ii)]
There is an isomorphism $u:V\rightarrow \bR^4$ such that $\left\{ u^*\tilde{\omega}_{\pi(i)}|i=1,\ldots,k\right\}$ is a basis of $W$.
\item[(iii)]
$H$ is definite.
\item[(iv)]
There exists a Euclidean metric and an orientation on $V$ such that $W$ is a subspace of the space of all self-dual two-forms on $V$.
\end{enumerate}
\end{lemma}
\begin{proof}
Condition (i) implies Condition (ii) by \cite[Theorem 3.1]{W} and \cite[Theorem 3.2]{W}. The converse direction follows since $\tilde{\omega}_i\wedge \tilde{\omega}_j=0$ for $i\neq j$ and so $\omega^2\neq 0$ for all $\omega\in W\backslash \{0\}$ if $\left\{ u^*\tilde{\omega}_{\pi(i)}|i=1,\ldots,k\right\}$ is a basis of $W$. Since $\tilde{\omega}_1,\, \tilde{\omega}_2,\,\tilde{\omega}_3$ form a basis of the self-dual two-forms on $\bR^4$ with respect to the standard Euclidean metric and orientation, we get the equivalence of (ii) and (iv).
To prove the equivalence of (i) and (iii), let $\omega=\sum_{i=1}^k a_i \omega_i\in W$ with $a:=(a_1,\ldots,a_k)^t\neq 0$. By Remark \ref{re:alginv}, $\omega$ has length two if and only if
\begin{equation*}
0\neq \omega^2=\sum_{i,j=1}^k a_i h_{ij} a_j \tau=a^t H a\, \tau,
\end{equation*} 
i.e. if and only if $a^t Ha\neq 0$. Hence, all elements in $W\backslash\{0\}$ have length two if and only if $H$ is definite.
\end{proof}
Now we are able to prove the main result of this subsection.
\begin{proposition}\label{pro:Hodgedualbytwoforms}
Let $V$ be a seven-dimensional vector space and $V=V_4\oplus V_3$ be a vector space decomposition of $V$ into a four-dimensional vector space $V_4$ and into a three-dimensional vector space $V_3$. Fix $\tau\in \L^4 V_4^*\backslash \{0\}$. Let $k\in \{0,1,2,3\}$ and $\omega_i\in \L^2 V_4^*$ for $i=1,\ldots,k$ be such that the symmetric matrix $H=(h_{ij})_{ij}\in \bR^{k\times k}$ defined by
\begin{equation*}
h_{ij}\tau=\omega_i\wedge \omega_j
\end{equation*}
is definite, where $k=0$ means that there is no condition. Then $V$ admits two-forms $\omega_{k+1},\ldots,\omega_3\in \L^2 V_4^*$ such that for all bases $\nu_1,\ldots,\nu_3\in \L^2 V_3^*$ of $\L^2 V_3^*$ the four-form
\begin{equation}\label{eq:G2HodgeDualbytwoforms}
\Psi:=\frac{1}{2}\omega_1^2+\sum_{i=1}^3 \omega_i\wedge \nu_i
\end{equation}
is the Hodge Dual of a $\G_2$-structure on $V$ and $V=V_4\op V_3$ is an adapted splitting.
\end{proposition}
\begin{proof}
Let $\tilde\omega_1:=e^{12}+e^{34}\in \L^2 \left(\bR^4\right)^*,\, \tilde\omega_2:=e^{13}-e^{24}\in \L^2 \left(\bR^4\right)^*,\, \tilde\omega_3:=e^{14}+e^{23}\in \L^2 \left(\bR^4\right)^*$. By Lemma \ref{le:subspaceslengthtwo}, there exists an isomorphism $u:V_4\rightarrow \bR^4$ such that $u^*\tilde{\omega}_1,\ldots,\linebreak u^*\tilde{\omega}_k$ is a basis of $\spa{\omega_1,\ldots,\omega_k}$. Since there is an automorphism of $V_4$ mapping $u^*\tilde{\omega}_1$ onto $\omega_1$, we may, without loss of generality, assume that $\omega_1=u^*\tilde{\omega}_1$. Let $A\in \bR^{k\times k}$, $A=(a_{ij})_{ij}$ be such that $\omega_j=\sum_{i=1}^k a_{ij} \left(u^*\tilde{\omega}_i\right)$ for $j=1,\ldots, k$. Set $f_i:=u^{-1}(e_i)\in V_4$ for $i=1,\ldots,k$ and set $\omega_l:=u^*\tilde\omega_l$ for $l=k+1,\ldots,3$. Since $\nu_1,\ldots,\nu_3$ is a basis, also $\tilde{\nu}_1,\ldots,\tilde{\nu}_3$ with $\tilde{\nu}_j=\sum_{i=1}^k a_{ji}\, \nu_i$ for $j=1,\ldots, k$, $\tilde{\nu}_j:=\nu_j$ for $j=k+1,\ldots,3$ is a basis of $V_3^*$. Thus, there exists a basis $f_5,\,f_6,\,f_7$ of $V_3$ such that $\tilde{\nu}_1=f^{56}$, $\tilde{\nu}_2=f^{67}$ and $\tilde{\nu}_3=f^{57}$ and we can compute
\begin{equation*}
\begin{split}
\Psi&=\frac{1}{2}\omega_1^2+\sum_{i=1}^3 \omega_i\wedge \nu_i=f^{1234}+\sum_{i,j=1}^k a_{ji} \left(u^*\tilde{\omega}_j\right)\wedge \nu_i+\sum_{i=k+1}^3 u^*\tilde{\omega}_i\wedge \tilde\nu_i\\
&=f^{1234}+\sum_{j=1}^3 \left(u^*\tilde{\omega}_j\right)\wedge \tilde{\nu}_j\\&=f^{1234}+f^{1256}+f^{3456}+f^{1367}-f^{2467}+f^{1457}+f^{2357}
\end{split}
\end{equation*}
and we see that $\Psi$ is the Hodge Dual of a $\G_2$-structure with adapted basis $(f_1,f_2,\ldots,f_7)$.
\end{proof}
\begin{remark}
The assertion of Proposition \ref{pro:Hodgedualbytwoforms} has been used implicitly in the literature several times before, cf. e.g. \cite{Br2} and \cite{Ma}.
\end{remark}
\section{Classification Results}\label{sec:class}
\subsection{Existence}\label{subsec:existence}
In this subsection, we state different existence results which are used in the Subsections \ref{subsec:firstcase} - \ref{subsec:fourthcase} to prove Theorem \ref{th:main}.
We begin with a general proposition which is true for any seven-manifold. This proposition is used afterwards to derive different more specific existence results for left-invariant cocalibrated $\G_2$-structures on Lie groups.
\begin{proposition}\label{pro:split}
Let $M$ be a seven-dimensional manifold. Assume that there exists a $\G_2$-structure $\varphi$ on $M$ which admits an adapted splitting $TM=E_4\op E_3$ such that the following is true:
\begin{enumerate}
\item[(i)]
$\Omega_1:=(\star_{\varphi} \varphi)|_{E_4}\in \Gamma\left(\L^4 E_4^*\right)\cong \Gamma\left(\L^4 \Ann{E_3}\right)\subseteq \Gamma(\Lambda^4 T^* M)$ is closed.
\item[(ii)]
There exists a bounded four-form $\Phi\in \Gamma\left(\L^3 \Ann{E_3}\wedge \Ann{E_4}\right)$ (i.e. $\left\|\Phi\right\|_{C_0}<\infty$) with $d\Phi=d\Omega_2$ for the four-form $\Omega_2:=\star_{\varphi} \varphi- \Omega_1\in \Gamma\left(\L^2\Ann{E_3}\wedge \L^2 \Ann{E_4}\right)$.
\end{enumerate}
Then $M$ admits a cocalibrated $\G_2$-structure, e.g. each $\G_2$-structure $\varphi_{\lambda}\in \Omega^3 (M)$ whose Hodge dual is given by
\begin{equation*}
\Psi_{\lambda}:=\lambda^4 \Omega_1 + \lambda^2 \Omega_2-\lambda^2 \Phi
\end{equation*}
for $\lambda\in \bR$ with $\left|\lambda\right| >\frac{\left\|\Phi\right\|_{C_0}}{\epsilon_0}$. Here, $\epsilon_0$ is the constant in Lemma \ref{le:universalconst}
\end{proposition}
\begin{proof}
Let $p\in M$. By Lemma \ref{le:AdaptedSplitting}, $(\Omega_2)_p\in \L^2 \Ann{(E_3)_p} \wedge \L^2 \Ann{(E_4)_p}$, $\sigma_{\lambda}:=\lambda^4(\Omega_1)_p+\lambda^2 (\Omega_2)_p$ is the Hodge-Dual of a $\G_2$-structure on $T_p M$ for all $\lambda\neq 0$ and $\left\| \lambda^3 \Phi_p \right\|_{\lambda}=\left\|\Phi_p\right\|_1=\left\|\Phi_p\right\|_{\varphi_p}$ for all $\lambda\neq 0$, where $\left\|\cdot\right\|_{\lambda}$ is the norm on $T_p M$ induced by $\sigma_{\lambda}$. Thus,
\begin{equation*}
\left\|(\Psi_{\lambda})_p-\sigma_{\lambda} \right\|_{\lambda}=\left\|\lambda^2 \Phi_p \right\|_{\lambda}=\frac{ \left\|\Phi_p\right\|_{\varphi_p}}{|\lambda|} \leq \frac{\left\|\Phi\right\|_{C^0}}{|\lambda|}<\epsilon_0
\end{equation*}
for all $\left|\lambda\right|>\frac{\left\|\Phi\right\|_{C^0}}{\epsilon_0}$. Hence, Lemma \ref{le:universalconst} shows that $\Psi_{\lambda}$ is the Hodge dual of a $\G_2$-structure on $M$. The assertion follows since $\Psi_{\lambda}$ is closed by construction.
\end{proof}
\begin{remark}\label{re:existenceforG2star}
\begin{itemize}
\item
The condition on the boundedness of $\Phi$ is trivially fulfilled if $\Phi$ is left-invariant or $M$ is compact. Moreover, if the initial $\G_2$-structure $\varphi$, the splitting $E_4\op E_3$ and $\Phi$ are left-invariant, so is the induced cocalibrated $\G_2$-structure.
\item
To prove an analogue of Proposition \ref{pro:split} in the left-invariant case for $\G_2$-structures, we do not need at all a metric. We only need that the orbit of all Hodge duals is open. For the proof, let $\g$ be a seven-dimensional Lie algebra $\g$. The openness of the orbit implies that for any sequence $(A_n)_n$, $A_n\in \GL(\g)$, any Hodge dual $\Psi\in \L^4 \g^*$ and any sequence $(\Phi_n)_n$, $\Phi_n\in \L^4 \g^*$ with $\lim\limits_{n\rightarrow \infty} \Phi_n=0$ there is $N\in \mathbb{N}$ such that for all $n\geq N$ the four-form $\Psi+\Phi_n$ and so also the four-form $A_n^*(\Psi+\Phi_n)$ is a Hodge dual of a $\G_2$-structure. Let now $\varphi\in \L^3 \g^*$ be a $\G_2$-structure and $\g=E_4\op E_3$ be a splitting into a four-dimensional subspace $E_4$ and a three-dimensional subspace $E_3$ such that $\Psi:=\star_{\varphi}\varphi=\Omega_1+\Omega_2$ with $\Omega_1\in \L^4 E_4^*$, $\Omega_2\in \L^2 E_4^*\wedge \L^2 E_3^*$, $d\Omega_1=0$ and such that there exists $\Phi\in \L^3 E_4^*\wedge  E_3^*$ with $d\Omega_2=d\Phi$. Here, we identify, as usual, $E_4^*\cong \Ann{E_3}$ and $E_3^*\cong \Ann{E_4}$ via the decomposition $\g=E_4\op E_3$. Define $A_n\in \GL(\g)$ such that it acts by multiplication with $n$ on $E_4$ and by the identity map on $E_3$ and set $\Phi_n:=-\frac{\Phi}{n}\in \L^3 E_4^*\wedge  E_3^*$. Then our previous considerations show that
\begin{equation*}
\Psi_n:=A_n^*(\star_{\varphi}\varphi+\Phi_n)=A_n^*\left(\Omega_1+\Omega_2-\frac{\Phi}{n}\right)=n^4 \Omega_1+n^2 \Omega_2-n^2 \Phi
\end{equation*}
is, for $n$ large enough, a Hodge dual of a $\G_2$-structure. Moreover, our assumptions imply that it is closed and so defines a cocalibrated $\G_2$-structure on $\g$. Note that we can literally transfer the proof to so-called cocalibrated $\G_2^*$-structures on Lie algebras and prove the analogous result for these structures.
\end{itemize}
\end{remark}
We apply Proposition \ref{pro:split} to the left-invariant case:
\begin{proposition}\label{pro:existence}
Let $\g=\g_4\op \g_3$ be a seven-dimensional Lie algebra which is the Lie algebra direct sum of a four-dimensional Lie algebra $\g_4$ and of a three-dimensional Lie algebra $\g_3$.
\begin{enumerate}
\item
If $\g_3$ is unimodular and there exists a $D:=h^2(\g_3)$-dimensional subspace $W$ of $\L^2 \g_4^*$ such that each non-zero element in $W$ is a symplectic two-form, then $\g$ admits a cocalibrated $\G_2$-structure.
\item
Let $\g_4\in\{A_{4,12},\mathfrak{r}_2\op\mathfrak{r}_2\}$. If $\g_3$ admits a contact-form $\alpha$, then $\g$ admits a cocalibrated $\G_2$-structure.
\item
If $\g_4$ is unimodular, admits a codimension one ideal $\uf$ isomorphic to $\mathfrak{h}_3$, $\g_3$ is not unimodular and $h^1(\g_4)+h^1(\g_3)-h^2(\g_4)\geq 2$, then $\g$ admits a cocalibrated $\G_2$-structure.
\end{enumerate}
\end{proposition}
\begin{proof}
\begin{enumerate}
\item
Choose a basis $\nu_1,\nu_2,\nu_3$ of $\L^2 \g_3^*$ such that $\nu_{D+1}=d\alpha_{D+1},\ldots,\nu_3=d\alpha_3$ is a basis of $d(\g_3^*)$, $\alpha_{D+1},\ldots,\alpha_3\in  \g_3^*$. Note that there are $3-D$ exact two-forms on $\g_3$ since the unimodularity of $\g_3$ is equivalent to the closure of all two-forms on $\g_3$. Furthermore, choose a basis $\omega_1,\ldots,\omega_D$ of $W$. Then Lemma \ref{le:subspaceslengthtwo} and Proposition \ref{pro:Hodgedualbytwoforms} imply that there exist two-forms $\omega_{D+1},\ldots,\omega_3\in \L^2 \g_4^*$ such that 
\begin{equation*}
\Psi:=\sum_{i=1}^3 \omega_i\wedge \nu_i+\frac{1}{2}\omega_1^2
\end{equation*}
is the Hodge dual of a $\G_2$-structure with adapted splitting $\g=\g_4\op \g_3$. Since $d(\L^2 \g_3^*)=0$, the identity $d(\sum_{i=1}^3 \omega_i\wedge \nu_i)=d(-\sum_{i=D+1}^3 d\omega_i\wedge \alpha_i)$ is true and $\sum_{i=D+1}^3 d\omega_i\wedge \alpha_i\in \L^3 \g_4^*\wedge  \g_3^*$. Hence, Proposition \ref{pro:split} implies the result.
\item
Let $e^1,e^2,e^3,e^4$ be a basis of $\g_4^*\in \{A_{4,12}^*,(\mathfrak{r}_2\op \mathfrak{r}_2)^*\}$ as in Table \ref{table:4d}, i.e. $de^1=e^{14}+e^{23}$, $de^2=e^{24}-\epsilon e^{13}$, $de^3=0=de^4$, where $\epsilon=1$ if $\g_4=A_{4,12}$ and $\epsilon=-1$ if $\g_4=\mathfrak{r}_2\op \mathfrak{r}_2$. Set $V_4^*:=\spa{e^4}\op \g_3^*$, $V_3^*:=\spa{e^1,e^2,e^3}$. Let $\alpha_1\in \g_3^*$ be a contact form and set $\omega_1:=2e^4\wedge \alpha_1-d\alpha_1\in \L^2 V_4^*$. Then $\omega_1^2\neq 0$ and $d\left(\frac{1}{2}\omega_1^2\right)=0$. Hence, if we set $\nu_1:=e^{12}$, $\nu_2:=e^{13}$, $\nu_3:=e^{23}$, Proposition \ref{pro:Hodgedualbytwoforms} implies the existence of two-forms $\omega_2,\omega_3\in \L^2 V_4^*$ such that
\begin{equation*}
\Psi:=\sum_{i=1}^3 \omega_i\wedge \nu_i+ \frac{1}{2}\omega_1^2
\end{equation*}
is the Hodge dual of a $\G_2$-structure with adapted splitting $\g=V_4\op V_3$. Decompose $\omega_i=e^4\wedge \alpha_i+\theta_i$ with $\alpha_i\in \g_3^*$, $\theta_i\in \L^2 \g_3^*$ for $i=2,3$. Then $d(\omega_1\wedge \nu_1)= d(2e^4\wedge \alpha_1\wedge e^{12}-d\alpha_1\wedge e^{12})=0$ and so
the differential of the four-form $\sum_{i=1}^3 \omega_i\wedge \nu_i$ is given by
\begin{equation*}
\begin{split}
d\left(\sum_{i=1}^3 \omega_i\wedge \nu_i \right)&=0+d(e^{134}\wedge \alpha_2+e^{234}\wedge \alpha_3)+d(e^{13}\wedge \theta_2+e^{23}\wedge \theta_3)\\
&=d(\epsilon e^{24}\wedge d\alpha_2-e^{14}\wedge d\alpha_3)\\
&\,\,\,\,\,\,+d(\epsilon(e^{24}\wedge \theta_2-e^2\wedge d\theta_2)-e^{14}\wedge \theta_3+e^1\wedge d\theta_3) \\
&=d(e^1\wedge \rho_1-\epsilon e^{2}\wedge \rho_2).
\end{split}
\end{equation*}
with $\rho_1:=-e^4\wedge (d\alpha_3+\theta_3)+d\theta_3 \in \Lambda^3 V_4^*,\, \rho_2:=-e^4\wedge (d\alpha_2+\theta_2)+d\theta_2 \in \Lambda^3 V_4^*$. Since $e^1\wedge \rho_1-\epsilon e^{2}\wedge \rho_2$ is in $ V_3^*\wedge \L^3 V_4^*$, Proposition \ref{pro:split} implies the result.
\item
By Lemma \ref{le:4dbasis} we may decompose $\g_4^*$ into $\spa{e^1}\op V_2\op \spa{e^4}$ for $e^1,e^4\in \g_4^*$ and a two-dimensional subspace $V_2$ such that $0\neq de^1\in \L^2 V_2$, $d\alpha=F(\alpha)\wedge e^4$ for all $\alpha\in V_2$, $F:V_2\rightarrow V_2$ a trace-free linear map, and $de^4=0$. Moreover, by Lemma \ref{le:3d} we may decompose $\g_3^*=W_2\op \spa{e^7}$ with $e^7\in \g_3^*$ and a two-dimensional subspace $W_2$ such that $d\beta=G(\beta)\wedge e^7$ for all $\beta\in W_2$, $G:W_2\rightarrow W_2$ a linear map which is not trace-free, and $de^7=0$. By rescaling $e^7$ we may assume that $\tr(G)=1$.

We have $\ker d|_{\L^2 \g_4^*}=\L^2 V_2\oplus V_2\wedge e^4\oplus \ker(F)\wedge e^1$. Thus, the identity
\begin{equation*}
2-\rk{F}+3=\dim{\ker(F)}+3=\dim{\ker d|_{\L^2 \g_4^*}}=h^2(\g_4)+4-h^1(\g_4)
\end{equation*}
is true. Moreover, $\dim{\ker G}=h^1(\g_3)-1$ and so the condition in the statement is equivalent to $\dim{\ker G}\geq 2-\rk{F}$. Hence, we may choose a basis $\alpha_1,\alpha_2$ of $V_2$, elements $\gamma_i\in V_2$, $1\leq i\leq \rk{F}$, and a basis $\beta_1,\beta_2$ of $W_2$ such that $de^1=\alpha_1\wedge\alpha_2$, such that $\alpha_i=F(\gamma_i)$, $1\leq i\leq \rk{F}$, is a basis of $F(V_2)$ and such that $\spa{\beta_j|\rk{F}+1\leq j\leq 2}$ is a subspace of $\ker G$. Set $V_4^*:=\spa{e^1}\op V_2\op \spa{e^7}$, $V_3^*:=W_2\op \spa{e^4}$ and
\begin{equation*}
\begin{split}
\nu_1:=& \beta_1\wedge \beta_2,\quad \nu_2:= \beta_1\wedge e^4,\quad \nu_3:=-\beta_2\wedge e^4,\\
\omega_1:=& e^{71}-de^1=e^{71}-\alpha_1\wedge \alpha_2,\quad \omega_2:= e^7\wedge  \alpha_2-e^1\wedge \alpha_1,\\
\omega_3:= & e^7\wedge \alpha_1+e^1 \wedge \alpha_2.
\end{split}
\end{equation*}
Since $\nu_1,\nu_2,\nu_3$ is a basis of $\L^2 V_3^*$, Proposition \ref{pro:Hodgedualbytwoforms} implies that
\begin{equation*}
\Psi:=\sum_{i=1}^3 \omega_i\wedge \nu_i + \frac{1}{2}\omega_1^2
\end{equation*}
is the Hodge dual of a $\G_2$-structure with adapted splitting $V_4\op V_3$. Moreover,
\begin{equation*}
\begin{split}
d(\omega_1\wedge \nu_1)& = d(e^{71}\wedge \beta_1\wedge \beta_2-de^1\wedge \beta_1\wedge \beta_2)\\
&=-e^{7}\wedge de^1\wedge\beta_1\wedge \beta_2+\tr(G) de^1\wedge \beta_1\wedge \beta_2\wedge e^7=0
\end{split}
\end{equation*}
and so
\begin{equation*}
\begin{split}
d\left(\sum_{i=1}^3 \omega_i\wedge \nu_i \right)&=d\left(-\sum_{i=1}^2 e^1\wedge \alpha_i\wedge \beta_i \wedge e^4\right)=-\sum_{i=1}^{\rk{F}} F(\gamma_i)\wedge e^4\wedge e^1\wedge G(\beta_i)\wedge e^{7}\\
&=d\left(-\sum_{i=1}^{\rk{F}} \gamma_i \wedge e^1 \wedge G(\beta_i )\wedge e^{7}\right).
\end{split}
\end{equation*}
But $-\sum_{i=1}^{\rk{F}} \gamma_i \wedge e^1 \wedge G(\beta_i)\wedge e^{7}$ is in $V_3^*\wedge \L^3 V_4^*$. Since $F$ is trace-free, $d(\L^4 V_4^*)=\{0\}$ and again Proposition \ref{pro:split} implies the result.
\end{enumerate}
\end{proof}
\begin{remark}
The following generalization of Proposition \ref{pro:existence} (a) follows from Proposition \ref{pro:split} using Lemma \ref{le:subspaceslengthtwo}:\\
Let $M=N\times G$ be a seven-dimensional manifold such that $N$ is a four-dimensional compact Riemannian manifold with trivial bundle of self-dual two-forms and such that $G$ is a unimodular three-dimensional Lie group. If $N$ admits $D:=h^2(\g)$ ($\g$ being the Lie algebra of $G$) self-dual, closed two-forms $\omega_i\in \Omega^2 N$ such that $\omega_i\wedge \omega_j=0$  and $\omega_i^2=\omega_j^2$ for $i\neq j$, then $M$ admits a cocalibrated $\G_2$-structure which is invariant under the left-action of $G$ on $M=N\times G$ given by left-translation on the second factor.
\end{remark}
$D=0$ is allowed in Proposition \ref{pro:existence} (a). Since each non-solvable four-dimensional Lie algebra $\g$ is a Lie algebra direct sum $\g=\h\op \bR$ with $\h\in\{\so(3),\so(2,1)\}$, $h^2(\so(3))=h^2(\so(2,1))=0$ and $\so(3),\so(2,1)$ are the only three-dimensional non-solvable Lie algebras, we get
\begin{corollary}\label{co:nonsolvable}
Let $\g=\g_4\op \g_3$ be a seven-dimensional Lie algebra which is the Lie algebra direct sum of a four-dimensional Lie algebra $\g_4$ and of a three-dimensional Lie algebra $\g_3$. If $\g$ is not solvable, then $\g$ admits a cocalibrated $\G_2$-structure.
\end{corollary}

\subsection{Obstructions}\label{sec:obstructions}\label{subsec:obstructions}
In this section, we derive obstructions to the existence of cocalibrated $\G_2$-structures on Lie algebras, which we use in subsections \ref{subsec:firstcase} - \ref{subsec:fourthcase} to prove Theorem \ref{th:main}. 

We start with
\begin{proposition}\label{pro:obstruct3dunimod}
Let $\g=\g_4\op \g_3$ be a seven-dimensional Lie algebra which is the Lie algebra direct sum of a four-dimensional Lie algebra $\g_4$ and of a three-dimensional unimodular Lie algebra $\g_3$ such that $\g_4$ admits a unique unimodular ideal $\uf$ of codimension one. If $\g$ admits a cocalibrated $\G_2$-structure, then
\begin{equation*}
h^1(\mathfrak{g}_4)+h^1(\mathfrak{u})-h^2(\mathfrak{g}_4)+h^2(\mathfrak{g}_3)\leq 4.
\end{equation*}
\end{proposition}
\begin{proof}
Let $\Psi$ be the Hodge dual of a cocalibrated $\G_2$-structure. Fix an element $e_4\in \g\backslash\uf$ and let $e^4\in \Ann{\uf}$ be such that $e^4(e_4)=1$. We set
\begin{equation*}
\L^{i,j,k}:=\L^i \uf^*\wedge \L^j \g_3^*\wedge \L^k \spa{e^4}
\end{equation*}
and denote by $\theta^{i,j,k}$ the projection of $\theta$ into $\L^{i,j,k}$ for all $i,j,k\in \mathbb{N}_0$ and all $(i+j+k)$-forms $\theta\in \L^{i+j+k} \g^*$. For the proof, we denote by $d$ the exterior differential on $\g$ and by $d_{\uf}$ the one on $\uf$. Lemma \ref{le:codimensiononeideal} implies the inclusions
\begin{equation*}
d(\L^{i,j,0})\subseteq \L^{i+1,j,0}\op \L^{i,j,1} \op \L^{i,j+1,0},\quad d(\L^{i,j,1})\subseteq \L^{i+1,j,1}\op \L^{i,j+1,1}
\end{equation*}
for all $i,j\in\mathbb{N}_0$ and the unimodularity of $\uf$ and $\g_3$ imply that for all $i\in \mathbb{N}_0$:
\begin{equation*}
\begin{split}
d(\L^{2,i,0}) &\subseteq \L^{2,i,1}\op \L^{2,i+1,0},\,\, d(\L^{2,i,1})\subseteq \L^{2,i+1,1},\\
d(\L^{i,2,0}) & \subseteq \L^{i+1,2,0}\op \L^{i,2,1},\,\, d(\L^{i,2,1})\subseteq \L^{i+1,2,1}.
\end{split}
\end{equation*}

We show that there are $D:=h^2(\g_3)$ linearly independent closed two-forms $\omega_1,\ldots,\omega_D\in \L^2\g_4^*$ such that $\spa{\omega_1,\ldots,\omega_D}\cap\, \L^{1,0,1}=\{0\}$. Note that
$\dim{\ker d|_{\L^{1,0,1}}}=h^1(\uf)$ since $\ker d|_{\L^{1,0,1}}=\ker d_{\uf}|_{ \uf^*} \wedge e^4$ by Lemma \ref{le:codimensiononeideal}. Hence, the existence of such $\omega_1,\ldots,\omega_D\in \L^2 \g_4^*$ implies
\begin{equation*}
\begin{split}
& h^2(\g_4)+4-h^1(\g_4)=\dim {\ker d|_{\L^2 \g^*}}\geq D+h^1(\uf)=h^2(\g_3)+h^1(\uf)\\
&\Leftrightarrow \, h^1(\mathfrak{g}_4)+h^1(\mathfrak{u})+h^2(\mathfrak{g}_3)-h^2(\mathfrak{g}_4)\leq 4.
\end{split}
\end{equation*} 
The two-forms $\omega_1,\ldots,\omega_D\in \L^2 \g_4^*$ will be certain parts of $\Psi^{2,2,0}+ \Psi^{1,2,1}$. Therefore, we decompose $\Psi$ as
\begin{equation*}
\Psi=\Omega+\rho\wedge e^4
\end{equation*}
with $\Omega\in \L^4 (\uf^*\op \g_3^*)$, $\rho\in \L^3 (\uf^*\op \g_3^*)$.

The first step of the proof is to show that the length of $\Omega^{2,2,0}$ is three. For that purpose, note that the identities
\begin{equation*}
0=(d\Psi)^{3,1,1}+(d\Psi)^{3,2,0}=d(\Omega^{3,1,0}), \quad 0=(d\Psi)^{1,3,1}+(d\Psi)^{2,3,0}=d(\Omega^{1,3,0})
\end{equation*}
are true. If $\g_4$ is not unimodular, then $d(\L^{3,0,0})=\L^{3,0,1}$. Hence, $\Omega^{3,1,0}=0$ in this case. If $\dim{[\g_4,\g_4]}=3$, then $d|_{\L^{1,0,0}}$ and so $d|_{\L^{1,3,0}}$ is injective and $\Omega^{1,3,0}=0$ follows. We know from Lemma \ref{le:4dcodim1} that the uniqueness of the unimodular ideal $\uf$ implies that $\g_4$ is not unimodular or $\dim{[\g_4,\g_4]}=3$. In both cases, Lemma \ref{le:alginv2} and the just obtained results show that then $l(\Omega^{2,2,0})=3$.

Next, we look at the $(2,2,1)$-component of $d\Psi$. This component is given by
\begin{equation*}
0=(d\Psi)^{2,2,1}=d(\Omega^{2,2,0})+d(\rho^{2,1,0}\wedge e^4)+d(\rho^{1,2,0}\wedge e^4)
\end{equation*}
Hence, $d(\Omega^{2,2,0}+\rho^{1,2,0}\wedge e^4)=-d(\rho^{2,1,0}\wedge e^4)\in \L^3 \g_4^* \wedge d(\g_3^*)$ and so $d(\Omega^{2,2,0}+\rho^{1,2,0}\wedge e^4)\in d(\L^2 \g_4^*)\wedge d(\g_3^*)$. Let
\begin{equation*}
\pi_k:\L^k \g_4^*\wedge \L^2 \g_3^*\rightarrow (\L^k \g_4^*\wedge \L^2 \g_3^*)/(\L^k \g_4^*\wedge d(\g_3^*))\cong \L^k \g_4^*\otimes H^2(\g_3)
\end{equation*}
be the natural projection for $k\in \mathbb{N}$, where the last canonical isomorphism holds since $\g_3$ is unimodular and so all two-forms on $\g_3$ are closed. Moreover, the identity $\pi_3\circ d=(d\otimes \id) \circ \pi_2$ is true. If we set $\Phi:=\pi_2(\Omega^{2,2,0}+\rho^{1,2,0}\wedge e^4)$, we get the identity
\begin{equation*}
(d\otimes \id)(\Phi)=\pi_3(d(\Omega^{2,2,0}+\rho^{1,2,0}\wedge e^4))=0.
\end{equation*}
Write
\begin{equation*}
\Phi=\sum_{i=1}^{D} \omega_i\otimes \nu_i
\end{equation*}
for $\omega_1,\ldots,\omega_D\in \L^2 \g_4^*$ and some basis $\nu_1,\ldots,\nu_D$ of $H^2(\g_3)$. Then $\omega_1,\ldots,\omega_D$ are all closed. By choosing a complement $V$ of $d(\g_3^*)$ in $\L^2 \g_3^*$, we may identify $\nu_1,\ldots,\nu_D$ with elements in $V$ and get
\begin{equation*}
\Omega^{2,2,0}=\psi+\sum_{i=1}^D \omega_i^{2,0,0}\wedge \nu_i
\end{equation*}
with $\psi\in \L^2 \uf^*\wedge d(\g_3^*)$. Since the length of $\Omega^{2,2,0}$ is three and the length of $\psi$ is at most $\dim{d(\g_3^*)}$, the length of $\sum_{i=1}^D \omega_i^{2,0,0}\wedge \nu_i$ has to be $3-\dim{d(\g_3^*)}=D$ and so $\omega_1^{2,0,0},\ldots,\omega_D^{2,0,0}$ have to be linearly independent. Thus, $\omega_1,\ldots, \omega_D$ are linearly independent and $\spa{\omega_1,\ldots,\omega_D}\cap  \L^{1,0,1}=\{0\}$. This finishes the proof.
\end{proof}
Proposition \ref{pro:obstruct3dunimod} gives us an obstruction if the three-dimensional part is unimodular, whereas the next proposition gives us an obstruction if the three-dimensional part is not unimodular.
\begin{proposition}\label{pro:obstruct3dnotunimod}
\begin{enumerate}
\item
Let $\g=\g_4\op \g_3$ be a seven-dimensional Lie algebra which is the Lie algebra direct sum of an almost Abelian four-dimensional Lie algebra $\g_4$ and of a three-dimensional non-unimodular Lie algebra $\g_3$. If $\g$ admits a cocalibrated $\G_2$-structure, then $\g_4$ is unimodular and $\g_3=\mathfrak{r}_2\op \bR$.
\item
Let $\g=\g_5\op \mathfrak{r}_2$ be a Lie algebra direct sum of a five-dimensional almost Abelian Lie algebra $\g_5$ and of the two-dimensional Lie algebra $\mathfrak{r}_2$. If $\g$ admits a cocalibrated $\G_2$-structure, then $\g_5$ is unimodular.
\end{enumerate}
\end{proposition}
\begin{proof}
\begin{enumerate}
\item
Let $\uf_3$ be an Abelian ideal in $\g_4$. Choose an element $e_4\in \g_4\backslash \uf_3$ and an element $e_7\in \g_3\backslash \uf_2$, where $\uf_2$ is a codimension one Abelian ideal in $\g_3$. Let $e^4\in \Ann{\uf_3}\subseteq \g_4^*$, $e^4(e_4)=1$ and $e^7\in \Ann{\uf_2}\subseteq \g_3^*$, $e^7(e_7)=1$. Let $\Psi\in \L^4 \g^*$ be the Hodge dual of a cocalibrated $\G_2$-structure, set $\L^{i,j,k,l}:=\L^i \uf_3^*\wedge \L^j \uf_2^*\wedge \L^k \spa{e^4}\wedge \L^l \spa{e^7}$ and denote by $\theta^{i,j,k,l}$ for each $s:=(i+j+k+l)$-form $\theta\in \L^s \g^*$ the projection of $\theta$ onto $\L^{i,j,k,l}$. By Lemma \ref{le:codimensiononeideal},
\begin{equation*}
d(\L^{i,j,k,l})\subseteq \L^{i,j,k+1,l}+\L^{i,j,k,l+1}
\end{equation*}
for all $i,j,k,l\in \mathbb{N}_0$. Let $\Omega$ be the part of $\Psi$ in $\L^4\left( \uf_3^*\oplus \uf_2^*\oplus\spa{e^4}\right)$, i.e.
\begin{equation*}
\Omega=\Psi^{2,2,0,0}+\Psi^{3,1,0,0}+\Psi^{3,0,1,0}+\Psi^{2,1,1,0}+\Psi^{1,2,1,0}.
\end{equation*}
By Lemma \ref{le:alginv1}, $r(\Omega)=1$. Hence, $l(\Psi^{2,2,0,0}+\Psi^{3,1,0,0})\geq 1$ and so $\Psi^{2,2,0,0}+\Psi^{3,1,0,0}\neq 0$. Moreover, the closure of $\Psi$ implies
\begin{equation*}
\begin{split}
0 &=(d\Psi)^{2,2,0,1}=d(\Psi^{2,2,0,0})^{2,2,0,1},\,\,\, 0=(d\Psi)^{3,1,1,0}=d(\Psi^{3,1,0,0})^{3,1,1,0},\\
0 &=(d\Psi)^{3,1,0,1}=d(\Psi^{3,1,0,0})^{3,1,0,1}.
\end{split}
\end{equation*}
Since $\g_3$ is not unimodular, we have $d(\L^2 \uf_2^*)=\L^3 \g_3^*$. Thus, $d(\Psi^{2,2,0,0})^{2,2,0,1}=0$ implies $\Psi^{2,2,0,0}= 0$ and so $\Psi^{3,1,0,0}\neq 0$. If $\g_4$ is non-unimodular, then $d(\L^3 \uf_3^*)=\L^4 \g_4^*$ and so $d(\Psi^{3,1,0,0})^{3,1,1,0}\neq 0$, a contradiction. Hence, $\g_4$ is unimodular. Similarly, if $d|_{\uf_2^*}$ is injective, then $d(\Psi^{3,1,0,0})^{3,1,0,1}\neq 0$, a contradiction. Thus, $d|_{\uf_2^*}$ is not injective and $\g_3=\mathfrak{r}_2\op \bR$.
\item
The proof of part (b) is completely analogous to (a). Therefore, let $\Psi\in \L^4 \g^*$ be the Hodge dual of a cocalibrated $\G_2$-structure, let $\uf$ be an Abelian ideal of dimension four in $\g_5$, $e_5\in \g_5\backslash \uf$, $e^5\in \Ann{\uf}\subseteq \g_5^*$ with $e^5(e_5)=1$ and $e^6,e^7$ a basis of $\mathfrak{r}_2^*$ such that $de^6=e^{67}$ and $de^7=0$. Similarly to (a), we set
\begin{equation*}
\L^{i,j,k,l}:=\L^i \uf^*\wedge \L^j \spa{e^6} \wedge \L^k \spa{e^5}\wedge \L^l \spa{e^7}
\end{equation*}
and denote for all $s:=(i+j+k+l)$-forms $\theta\in \L^s \g^*$ the projection of $\theta$ onto $\L^{i,j,k,l}$ by $\theta^{i,j,k,l}$. Then 
$d(\L^{i,j,k,l})\subseteq \L^{i,j,k+1,l}+\L^{i,j,k,l+1}$ as in (a). Moreover, arguing as in (a), we get $\Psi^{4,0,0,0}+\Psi^{3,1,0,0}\neq 0$. Since $de^6\neq 0$, the identity
\begin{equation*}
0=(d\Psi)^{3,1,0,1}=d(\Psi^{3,1,0,0})^{3,1,0,1}
\end{equation*}
is true only if $\Psi^{3,1,0,0}=0$. Thus, $\Psi^{4,0,0,0}\neq 0$. But then
\begin{equation*}
0=(d\Psi)^{4,0,1,0}=d(\Psi^{4,0,0,0})
\end{equation*}
implies that $\g_5$ is unimodular by Lemma \ref{le:codimensiononeideal}.
\end{enumerate}
\end{proof}
\subsection{$\g_4$ not unimodular, $\g_3$ unimodular}\label{subsec:firstcase}
In this subsection, we prove Theorem \ref{th:main} (a). In the following, $\g=\g_4\op \g_3$ always denotes a seven-dimensional Lie algebra which is the Lie algebra direct sum of a four-dimensional non-unimodular Lie algebra $\g_4$ and of a three-dimensional unimodular Lie algebra $\g_3$. Furthermore, $\uf$ denotes the unimodular ideal of $\g_4$.

Proposition \ref{pro:obstruct3dunimod} shows that if $h^1(\mathfrak{g}_4)+h^1(\mathfrak{u})+h^2(\mathfrak{g}_3)-h^2(\mathfrak{g}_4)>4$, then $\g$ does not admit a cocalibrated $\G_2$-structure, giving us one direction of Theorem \ref{th:main} (a).

For the other direction, Lemma \ref{le:symplect} and Proposition \ref{pro:existence} (a) tell us that if $h^1(\mathfrak{g}_4)+h^1(\mathfrak{u})+h^2(\mathfrak{g}_3)-h^2(\mathfrak{g}_4)\leq 4$ and $\uf\neq e(1,1)$, then $\g$ does admit a cocalibrated $\G_2$-structure. By Table \ref{table:4d} or by Remark \ref{re:4dLAs}, the only four-dimensional non-unimodular Lie algebra $\g_4$ with unimodular ideal $\uf=e(1,1)$ is $\g_4=\mathfrak{r}_2\op \mathfrak{r}_2$. For $\g_4=\mathfrak{r}_2\op \mathfrak{r}_2$, Lemma \ref{le:contact} and Proposition \ref{pro:existence} (b) imply that $\g_4\op \g_3=\mathfrak{r}_2\op \mathfrak{r}_2\op \g_3$ does admit a cocalibrated $\G_2$-structure if $\g_3\neq \bR^3$, i.e. if $h^2(\g_3)\leq 2$. But $h^1(\mathfrak{r}_2\op \mathfrak{r}_2)+h^1(e(1,1))-h^2(\mathfrak{r}_2\op \mathfrak{r}_2)=2$. Hence, also in this case, $\g_4\op \g_3$ admits a cocalibrated $\G_2$-structure if and only if $h^1(\mathfrak{g}_4)+h^1(\mathfrak{u})+h^2(\mathfrak{g}_3)-h^2(\mathfrak{g}_4)\leq 4$. This proves Theorem \ref{th:main} (a).

\subsection{$\g_4$ unimodular, $\g_3$ unimodular}\label{subsec:secondcase}
Here, we prove Theorem \ref{th:main} (b) and denote by $\g=\g_4\op \g_3$ always a seven-dimensional Lie algebra which is the Lie algebra direct sum of a four-dimensional unimodular Lie algebra $\g_4$ and of a three-dimensional unimodular Lie algebra $\g_3$.

We begin with the case that $\g_4$ is indecomposable.
If $[\g_4,\g_4]=\bR^3$, then Lemma \ref{le:symplect}, Proposition \ref{pro:existence} (a) and Proposition \ref{pro:obstruct3dunimod} tell us that $\g$ admits a cocalibrated $\G_2$-structure if and only if
\begin{equation*}
h^1(\g_4)+3- h^2(\g_4)+ h^2(\g_3) =h^1(\g_4)+h^1(\bR^3)-h^2(\g_4)+h^2(\g_3)\leq 4.
\end{equation*}
Table \ref{table:4d} tells us that always $h^1(\g_4)-h^2(\g_4)=1$ in the considered cases. Hence, $\g$ admits for these cases a cocalibrated $\G_2$-structure exactly when $h^2(\g_3)=0$, i.e. when $\g_3\in \{\so(3),\so(2,1)\}$.

Next, we assume that $\g_4$ is indecomposable but $[\g_4,\g_4]\neq \bR^3$. By inspection of Table \ref{table:4d}, $\g_4\in \{A_{4,1},A_{4,8},A_{4,10}\}$.

Let us begin with $\g_4\in \{A_{4,8},A_{4,10}\}$. Then, in both cases, $h^1(\mathfrak{g}_4)+h^1(\mathfrak{u})-h^2(\mathfrak{g}_4)=3$, where $\uf$ is the unique unimodular ideal in $\g_4$ which is isomorphic to $\h_3$. Thus, Proposition \ref{pro:obstruct3dunimod} yields that $\g$ does not admit a cocalibrated $\G_2$-structure if $h^2(\g_3)\geq 2$. Conversely, Corollary \ref{co:nonsolvable} tells us that if $h^2(\g_3)=0$, i.e. $\g_3$ is not solvable, then $\g$ does admit a cocalibrated $\G_2$-structure. So we are left with the case that $h^2(\g_3)=1$, i.e. $\g_3\in\{e(2),e(1,1)\}$. For $\g=A_{4,8}\op e(1,1)$, a cocalibrated $\G_2$-structure is given in Table \ref{table:examples}.
All other cases do not admit a cocalibrated $\G_2$-structure:
\begin{lemma}
Let $\g\in \{A_{4,8}\op e(2), A_{4,10}\op e(2), A_{4,10}\op e(1,1)\}$. Then $\g$ does not admit a cocalibrated $\G_2$-structure.
\end{lemma}
\begin{proof}
Let $e^1,e^2,e^3,e^4$ be the basis of $\g_4^*$, $\g_4\in \{A_{4,8},A_{4,10}\}$ as in Table \ref{table:4d}. Then there exists a linear, trace-free, invertible map $F:\spa{e^2,e^3}\rightarrow \spa{e^2,e^3}$ such that $de^1=e^{23}$, $d\alpha=F(\alpha)\wedge e^4$, $de^4=0$ for all $\alpha\in \spa{e^2,e^3}$. For $\g_4=A_{4,8}$ we have $F(e^2)=e^2$, $F(e^3)=-e^3$ whereas for $\g_4=A_{4,10}$ we have $F(e^2)=e^3$ and $F(e^3)=-e^2$. In particular, $\det(F)=-1$ if $\g_4=A_{4,8}$ and $\det(F)=1$ if $\g_4=A_{4,10}$.

Let $e^5,e^6,e^7$ be a basis of $\g_3^*$, $\g_3\in \{e(2),e(1,1)\}$ as in Table \ref{table:3d}. Then there exists a linear, trace-free, invertible map $G:\spa{e^5,e^6}\rightarrow \spa{e^5,e^6}$ such that $d\beta=G(\beta)\wedge e^7$, $de^7=0$ for all $\beta\in \spa{e^5,e^6}$. In both cases we have $G(e^5)=e^6$, whereas $G(e^6)=e^5$ if $\g_3=e(1,1)$ and $G(e^6)=-e^5$ if $\g_3=e(2)$. In particular, $\det(G)=-1$ if $\g_3=e(1,1)$ and $\det(G)=1$ if $\g_3=e(2)$.

Let us now assume that $\Psi\in \L^4\g^*$ is a (closed) Hodge dual of a cocalibrated $\G_2$-structure $\varphi\in \L^3 \g^*$. We decompose $\Psi$ uniquely into
\begin{equation*}
\Psi=\rho\wedge e^1+\Omega
\end{equation*}
with $\rho\in \L^3 (\spa{e^2,e^3,e^4}\op \g_3^*)$, $\Omega\in \L^4 (\spa{e^2,e^3,e^4}\op \g_3^*)$. Then
\begin{equation*}
0=d\Psi=d\rho\wedge e^1- \rho\wedge e^{23}+d\Omega,
\end{equation*}
$d\Omega\in \L^3 \spa{e^2,e^3,e^5,e^6}\wedge e^{47}$ (note that $de^{2356}=0$) and $d\rho\in \L^4 (\spa{e^2,e^3,e^4}\op \g_3^*)$ imply $d\rho=0$
and $\pr_{\spa{e^{456},e^{567}}}(\rho)=0$. Moreover, $\ker F=\{0\}=\ker G$ and $d\rho=0$ imply $\pr_{\L^3 \spa{e^2,e^3,e^5,e^6}} (\rho)=0$.

Thus, $\rho=(\omega_1+a e^{23})\wedge e^4+(\omega_2+b e^{23})\wedge e^7+\beta\wedge e^{47}$ for certain $\omega_1,\omega_2\in  \spa{e^2,e^3}\wedge \spa{e^5,e^6}$, $a,b\in \bR$ and $\beta\in  \spa{e^2,e^3,e^5,e^6}$. Now Lemma \ref{le:alginv1} tells us that $r(\rho)=2$ and so Lemma \ref{le:alginv2} yields that $\omega_1+ae^{23}$ and $\omega_2+be^{23}$ span a two-dimensional subspace in $\L^2 \spa{e^2,e^3,e^5,e^6}$ in which each non-zero element has length two. This is equivalent to the requirement that $\omega_1$ and $\omega_2$ span such a two-dimensional subspace of $\L^2\spa{e^2,e^3,e^5,e^6}$ and Lemma \ref{le:subspaceslengthtwo} shows that this is equivalent to $\omega_1^2\neq 0$ and $C-B^2>0$ for the numbers $B,C\in \bR$ defined by $\omega_1\wedge \omega_2=B \omega_2^2$, $\omega_2^2=C \omega_1^2$. By Lemma \ref{le:subspaceslengthtwo}, there exists a basis $\alpha_1,\alpha_2$ of $\spa{e^2,e^3}$ and $\alpha_3,\alpha_4$ of $\spa{e^5,e^6}$ such that $\omega_1=\alpha_1\wedge \alpha_4+\alpha_2\wedge \alpha_3$. Since $d(\omega_1\wedge e^4+\omega_2\wedge e^7)=d\rho=0$, we must have $\omega_2=F^{-1}(\alpha_1)\wedge G(\alpha_4)+ F^{-1}(\alpha_2)\wedge G(\alpha_3)$. Thus, $C=\frac{\det(G)}{\det(F)}$. If $\g\in \{A_{4,8}\op e(2),A_{4,10}\op e(1,1)\}$, then $C<0$ leading to $C-B^2<0$. Thus, for these cases, there cannot exist a cocalibrated $\G_2$-structure.

For the missing case $\g=A_{4,10}\op e(2)$, let $\omega_1:=c_1 e^{25}+c_2e^{26}+c_3 e^{35}+c_4 e^{36}$ be a general two-form in $\spa{e^2,e^3}\wedge  \spa{e^5,e^6}$ of length two, i.e. with $c_1 c_4-c_2 c_3\neq 0$. Then $\omega_2=-c_4 e^{25}+c_3 e^{26}+c_2 e^{35}-c_1 e^{36}$, $B=-\frac{c_1^2+c_2^2+c_3^2+c_4^2}{2(c_1 c_4-c_2 c_3)}$, $C=1$ and so
\begin{equation*}
\begin{split}
C-B^2 &= \frac{4(c_1 c_4-c_2 c_3)^2-(c_1^2+c_2^2+c_3^2+c_4^2)^2}{4(c_1 c_4-c_2 c_3)^2}\\
&=-\frac{((c_1-c_4)^2+(c_2+c_3)^2) ((c_1+c_4)^2+(c_2-c_3)^2)}{4(c_1 c_4-c_2 c_3)^2}<0.
\end{split}
\end{equation*}
Thus, $A_{4,10}\op e(2)$ does not admit a cocalibrated $\G_2$-structure.
\end{proof}

Next we consider direct sums with $A_{4,1}$. The Lie algebra $A_{4,1}$ is almost Abelian and it admits a symplectic two-form, e.g. $\omega=e^{14}+e^{23}$ in the basis $e^1,e^2,e^3,e^4$ given in Table \ref{table:4d}. Hence, Proposition \ref{pro:existence} (a) shows that $A_{4,1}\op \g_3$ admits a cocalibrated $\G_2$-structure if $h^2(\g_3)\leq 1$, i.e. if $\g_3\notin\{\bR^3,\h_3\}$. The Lie algebra $\g=A_{4,1}\oplus \bR^3$ is also almost Abelian. The almost Abelian Lie algebras admitting a cocalibrated $\G_2$-structure were determined in \cite{F} and the results there show that $\g=A_{4,1}\oplus \bR^3$ does not admit a cocalibrated $\G_2$-structure. Also $\g=A_{4,1}\oplus \h_3$ does not admit a cocalibrated $\G_2$-structure.
\begin{lemma}
Let $\g=A_{4,1}\op \h_3$. Then $\g$ does not admit a cocalibrated $\G_2$-structure.
\end{lemma}
\begin{proof}
Choose a basis $e^1,e^2,e^3,e^4,e^5,e^6,e^7$ of $A_{4,1}\oplus \h_3$ as in Table \ref{table:4d} and Table \ref{table:3d}, i.e.
\begin{equation*}
de^1=e^{24},\; de^2=e^{34},\; de^3=0,\; de^4=0,\; de^5= e^{67}\; de^6=0,\; de^7=0,
\end{equation*}
Assume that there exists a cocalibrated $\G_2$-structure and let
\begin{equation*}
\Psi=\sum_{1\leq i< j< k< l\leq 7} a_{ijkl} e^{ijkl}
\end{equation*}
be its (closed) Hodge dual. Then a short computation shows that $a_{1567}=a_{2567}=a_{1256}=a_{1356}=a_{1257}=a_{1357}=a_{1235}=0$.
If we decompose $\Psi$ uniquely into
\begin{equation*}
\Psi=\Omega+e^1\wedge \nu+e^{14}\wedge \omega,
\end{equation*}
with $\Omega\in \L^4 \spa{e^2,e^3,e^4,e^5,e^6,e^7}$, $\nu\in \L^3 \spa{e^2,e^3,e^5,e^6,e^7}$ and\\
 $\omega\in \L^2 \spa{e^2,e^3,e^5,e^6,e^7}$, then $\nu$ actually is in $\L^3 \spa{e^2,e^3,e^6,e^7}$ and so of length at most one. If we consider the decomposition $\bigl(\spa{e^2,e^3,e^5,e^6,e^7}\op \spa{e^4}\bigr)\op \spa{e^1}=\g^*$, Lemma \ref{le:alginv1} implies that the length of $\nu$ has to be at least two, a contradiction. 
\end{proof}
So we are left with the case that $\g_4$ is decomposable. Then $\g_4$ is the Lie algebra direct sum of a three-dimensional unimodular Lie algebra $\h$ and $\bR$ and $\g$ always admits a cocalibrated $\G_2$-structure.
\begin{proposition}\label{pro:decomposableunimodular}
Let $\g=\g_4\op\g_3$ be a Lie algebra direct sum of a four-dimensional unimodular Lie algebra $\g_4$ and of a three-dimensional unimodular Lie algebra $\g_3$. Moreover, let $\g_4=\h\op \bR$ be a Lie algebra direct sum of a three-dimensional unimodular Lie algebra $\h$ and $\bR$. Then $\g$ admits a cocalibrated $\G_2$-structure.
\end{proposition}
\begin{proof}
We may assume that $h^2(\h)\geq h^2(\g_3)$. Moreover, we may assume that $\g_4=\h\oplus \mathbb{R}$ does admit an Abelian ideal $\uf$ of codimension $1$ since otherwise $\h\in \{\mathfrak{so}(3),\mathfrak{so}(2,1)\}$ and Corollary \ref{co:nonsolvable} gives us the affirmative answer. By K\"unneth's formula, $h^1(\h\oplus \mathbb{R})=h^1(\h)+1$ and $h^2(\h\oplus \mathbb{R})=h^2(\h)+h^1(\h)$. Thus
\begin{equation*}
\begin{split}
h^1(\h\oplus \mathbb{R})+h^1(\uf)+h^2(\g_3)-h^2(\h\oplus \mathbb{R})&=h^1(\h)+1+3+h^2(\g_3)-h^2(\h)-h^1(\h)\\
&=h^2(\g_3)-h^2(\h)+4\leq 4,
\end{split}
\end{equation*}
and Proposition \ref{pro:existence} (a) implies the statement.
\end{proof}
\subsection{$\g_4$ unimodular, $\g_3$ not unimodular}\label{subsec:thirdcase}
In this subsection, we prove Theorem \ref{th:main} (c). In the following, $\g=\g_4\op \g_3$ always denotes a seven-dimensional Lie algebra which is the Lie algebra direct sum of a four-dimensional unimodular Lie algebra $\g_4$ and of a three-dimensional non-unimodular Lie algebra $\g_3$.

We start with the case that $\g_4$ is almost Abelian. Then Proposition \ref{pro:obstruct3dnotunimod} (a) implies that if $\g_3\neq \mathfrak{r}_2\op\bR$, then $\g$ does not admit a cocalibrated $\G_2$-structure. So, in this case, it remains to consider sums of the form $\g_4\op \mathfrak{r}_2\op \bR$. This is done in Theorem \ref{th:5d+r2} which tells us more generally when a direct sum of the form $\g=\h\op\mathfrak{r}_2$ where $\h$ is a five-dimensional almost Abelian Lie algebra possesses a cocalibrated $\G_2$-structure. For the proof of this theorem, we need the following
\begin{lemma}\label{le:5d+r2}
Let $\g=\g_5\op \mathfrak{r}_2$ be a Lie algebra direct sum of a five-dimensional unimodular almost Abelian Lie algebra $\g_5$ and $\mathfrak{r}_2$. Let $\af$ be an Abelian ideal of dimension four in $\g_5$. Choose $e_5\in \h\backslash\af$ and $e^5\in\Ann{\af}\subseteq \h^*$, $e^5(e_5)=1$. Then $\g_5$ admits a cocalibrated $\G_2$-structure if and only if there exist two linearly independent two-forms $\omega_1,\,\omega_2\in \L^2 \af^*\cong \L^2 \Ann{e_5}$ such that each non-zero linear combination is of length two and such that $d\omega_1=\omega_2\wedge e^5$.
\end{lemma}
\begin{proof}
Let $e^6,e^7$ be a basis of $\mathfrak{r}_2^*$ such that $de^6=e^{67}$, $de^7=0$. Assume first that $\g$ admits a cocalibrated $\G_2$-structure $\varphi\in \L^3 \g^*$ with (closed) Hodge dual $\Psi\in \L^4 \g^*$.
Decompose $\Psi$ uniquely into
\begin{equation*}
\Psi=\Omega + \rho\wedge e^6
\end{equation*}
with $\Omega\in \L^4 (\g_5^*\op \spa{e^7})$, $\rho\in \L^3 (\g_5^*\op \spa{e^7})$. Since $d\Omega\in \L^5(\g_5^*\op \spa{e^7})$ and $d(\rho\wedge e^6)\in \L^4(\g_5^*\op \spa{e^7})\wedge e^6$, the identities $d\Omega=0=d(\rho\wedge e^6)$ are true.

Set $\L^{i,j,k}:=\L^i  \af^*\wedge \L^j \spa{e^5}\wedge \L^k \spa{e^7}$. For an $s:=(i+j+k)$-form $\theta\in \L^s(\g_5^*\op \spa{e^7})$ let $\theta^{i,j,k}$ be the projection of $\theta$ onto $\L^{i,j,k}$. Lemma \ref{le:codimensiononeideal} implies $d(\L^{i,0,k})\subseteq \L^{i,1,k}$ and $d(\L^{i,1,k})=0$ for all $i,k\in \mathbb{N}_0$.

The closure of $\rho\wedge e^6$ implies $0=d(\rho\wedge e^6)=d\rho\wedge e^6-\rho\wedge e^{67}$ and so $0=d\rho+\rho\wedge e^7$. Then the identities
\begin{equation*}
0=(d\rho+\rho\wedge e^7)^{3,0,1}=\rho^{3,0,0}\wedge e^7,\,
0=(d\rho+\rho\wedge e^7)^{2,1,1}= d(\rho^{2,0,1})+\rho^{2,1,0}\wedge e^7 
\end{equation*}
are true. Thus, $\rho^{3,0,0}=0$ and $d(\rho^{2,0,1})=-\rho^{2,1,0}\wedge e^7$. This shows that
\begin{equation*}
\rho=\omega_1\wedge e^7-\omega_2\wedge e^5+\alpha\wedge e^{57}
\end{equation*}
for $\omega_1,\,\omega_2 \in \L^{2,0,0}$, $\alpha\in \L^{1,0,0}$ and that
\begin{equation*}
\omega_2\wedge e^{57}=-\rho^{2,1,0}\wedge e^7=d(\rho^{2,0,1})=d(\omega_1\wedge e^7)=d\omega_1\wedge e^7\,\, \Leftrightarrow\,\,
d\omega_1=\omega_2\wedge e^5.
\end{equation*}
By Lemma \ref{le:alginv1}, $r(\rho)=2$ and Lemma \ref{le:alginv2} yields that $V:=\spa{\omega_1,\omega_2}$ is two-dimensional and each non-zero element in $V$ has length two.

Conversely, let $\omega_1,\,\omega_2\in \L^2 \af^*$ be such that $d\omega_1=\omega_2\wedge e^5$ and such that $\omega_1$, $\omega_2$ are linearly independent and each non-zero linear combination of them is of length two. Set $V_4:=\af^*$, $V_3:=\spa{e^5}\op \mathfrak{r}_2^*$, $\nu_1:=e^{67}\in \L^2 V_3$, $\nu_2:=e^{56}\in \L^2 V_3$, $\nu_3:=e^{57}\in \L^2 V_3$. By Lemma \ref{le:subspaceslengthtwo} and Proposition \ref{pro:Hodgedualbytwoforms}, there exists a two-form $\omega_3\in \L^2 \af^*$ such that
\begin{equation*}
\Psi:=\sum_{i=1}^3 \omega_i\wedge \nu_i+\frac{1}{2}\omega_1^2
\end{equation*}
is the Hodge dual of a $\G_2$-structure. By Lemma \ref{le:codimensiononeideal}, $d(\L^4 \af^*)=0$ and $d(\L^k\af^*\wedge e^5)=0$ for all $k\in \mathbb{N}_0$. Using these properties of $d$, a short computation shows that $\Psi$ is closed.
\end{proof}
Lemma \ref{le:5d+r2} allows us to prove
\begin{theorem}\label{th:5d+r2}
Let $\g=\h\op \mathfrak{r}_2$ be a Lie algebra direct sum of a five-dimensional almost Abelian Lie algebra $\h$ and of $\mathfrak{r}_2$. Then $\g$ admits a cocalibrated $\G_2$-structure if and only if $\h$ is unimodular and $\h\notin \left\{\bR^5, \h_3\op \bR^2, A_{5,7}^{-1/3,-1/3, -1/3}\right\}$.
\end{theorem}
\begin{proof}
By Proposition \ref{pro:obstruct3dnotunimod} (b), $\h$ has to be unimodular if $\g$ admits a cocalibrated $\G_2$-structure.
So, for the rest we assume that $\h$ is unimodular and let $e_5\in \h\backslash \mathfrak{a}$, $e^5\in \Ann{\mathfrak{a}}\subseteq \h^*$, $e^5(e_5)=1$. By Lemma \ref{le:codimensiononeideal}, there exists a linear trace-free map $H:\af^*\rightarrow \af^*$ such that $d \alpha=H(\alpha)\wedge e^5$, $de^5=0$ for all $\alpha\in \af^*$. Let $e^6,e^7$ be a basis of $\mathfrak{r}_2^*$ with $de^6=e^{67}$, $de^7=0$. Then Lemma \ref{le:5d+r2} tells us that $\g$ admits a cocalibrated $\G_2$-structure if and only if there are two linearly independent two-forms $\omega_1,\,\omega_2\in \L^2 \mathfrak{a}^*$ such that $d\omega_1=\omega_2\wedge e^5$ and such that each non-zero linear combination is of length two.

We first prove that such a pair of two-forms always exists if there is a vector decomposition $\af^*=V_2\op W_2$ into two two-dimensional $H$-invariant subspaces such that the restrictions of $H$ to $V_2$ and to $W_2$ are both not a multiple of the identity. In this case, we may choose for each $\lambda\neq 0$ a basis $e^1,e^2$ of $V_2$ and a basis $e^3,e^4$ of $W_2$ such that the restrictions of $H$ to $V_2$ and $W_2$ with respect to the corresponding bases are given by
\begin{equation*}
\begin{pmatrix}
 0 & -\frac{\det(H|_{V_2})}{\lambda} \\
 \lambda  & \tr(H|_{V_2})
\end{pmatrix}
\, \textrm{ and } \,
\begin{pmatrix}
\tr(H|_{W_2}) & - \lambda \\
 \frac{\det(H|_{W_2})}{\lambda} & 0
\end{pmatrix},
\end{equation*}
respectively. Set $\omega_1:=e^{14}+e^{23}$. Then $\omega_1$ is of length two and $d\omega_1=\bigr(\lambda(e^{13}-e^{24})+\omega_3\bigl)\wedge e^5$ with $\omega_3:=d e^{23}\in \L^2 \af^*$. Set $\omega_2:=\lambda(e^{13}-e^{24})+\omega_3$ and observe that $d\omega_1=\omega_2\wedge e^5$ and
\begin{equation*}
\omega_1\wedge \omega_2=e_5\hook(\omega_1\wedge d\omega_1)=e_5\hook\left(d\left(\frac{1}{2}\omega_1^2\right)\right)=0
\end{equation*}
since $\g_5$ is unimodular.
Furthermore, observe that $C(\lambda)$, defined by
\begin{equation*}
\omega_2^2=\lambda^2 e^{1234}+2\lambda (e^{13}-e^{24})\wedge \omega_3+\omega_3^2=C(\lambda)\omega_1^2,
\end{equation*}
fulfills $C(\lambda)=\lambda^2+ \mathcal{O}(\lambda)$ as $\lambda\rightarrow \infty$. Thus, for $|\lambda|$ sufficiently large,
$C(\lambda)>0$ and Lemma \ref{le:subspaceslengthtwo} tells us that $\omega_1,\omega_2$ span a two-dimensional subspace in which each non-zero element is of length two. So all considered Lie algebras which admit such a splitting do admit a cocalibrated $\G_2$-structure.

Next, we assume that $\af^*$ does not admit a splitting as above and look at the possible real Jordan normal forms of $H$. Therefore, we denote by $J_m(a)\in \bR^{m\times m}$ the matrix consisting of one Jordan block of size $m$ with $a\in\bR$ on the diagonal and the $1$s are on the superdiagonal and by $M_{b,c}$ the real two-by-two matrix $\begin{pmatrix} b & c \\ -c & b \end{pmatrix}$. We get, after rescaling $e^5$, that there is a basis $e^1,e^2,e^3,e^4$ of $\af^*$ such that $H$ acts with respect to this basis as one of the following matrices:
\begin{equation*}
\begin{split}
& \begin{pmatrix} J_3(a) & A \\ 0 & -3a \end{pmatrix},\;\; \begin{pmatrix} M_{0,1} & I_2 \\ 0 & M_{0,1} \end{pmatrix},\;\;  \diag(M_{b,1},-b,-b),\;\; \diag(J_2(c),-c,-c),\\
& \diag\left(f,-f/3,-f/3,-f/3\right),\quad 
 a,\, c,\, f,\, A\in \{0,1\},\, b\in \bR^+.
\end{split}
\end{equation*}
In the first case, $\omega_1:=e^{12}+e^{34}-5 e^{23}$ and $\omega_2:=-e^{24}+2a (-e^{12}+e^{34})+10 a e^{23}+5 e^{13}$ fulfil all desired conditions. In the second case, we may choose $\omega_1:=e^{12}-e^{34}$ and $\omega_2:=e^{14}-e^{23}$ and in the third case, $\omega_1:=e^{13}-e^{24}$ and $\omega_2:=e^{14}+e^{23}$ do the job. In the fourth case, we start with $c=1$. Then $\omega_1:=e^{13}-e^{24}- \frac{1}{2}\left(e^{12}-e^{34}\right)$, $\omega_2:=e^{12}+e^{34}+e^{14}$ fulfil all desired conditions. If $c=0$, then $\h=\h_3 \op \bR^2$ and we already know by Proposition \ref{pro:obstruct3dunimod} that $\g= \mathfrak{r}_2\op \bR^2 \op \h_3$ does not admit a cocalibrated $\G_2$-structure. However, this also follows easily from the fact that in this case $d(\L^2 \af^*)=\spa{ e^{135},e^{145} }$. In the last case, let $\omega_1\in \L^2 \af^*$ be of length two. Then there exist $\alpha\in  \spa{e^2,e^3,e^4}$ and $\omega\in \L^2 \spa{e^2,e^3,e^4}$ such that $\omega_1=\omega+\alpha\wedge e^1$. But then $d\omega_1=\frac{2}{3} f \left(\omega-\alpha\wedge e^1\right)\wedge e^5$, i.e. $\omega_2=\frac{2}{3} f\left(\omega-\alpha\wedge e^1\right)$ and so $\frac{2}{3} f \omega_1+\omega_2=\frac{4}{3} f \omega$ is of length one. Thus, $\g$ does not admit a cocalibrated $\G_2$-structure in this case, i.e. if $\h\in \left\{\bR^5,A_{5,7}^{-1/3,-1/3, -1/3}\right\}$.
\end{proof}
\begin{remark}
We like to note an interesting consequence of Theorem \ref{th:5d+r2}. It is well-known, cf. e.g. \cite{Sto}, that a half-flat $\SU(3)$-structure on $\mathfrak{b}$ naturally induces a cocalibrated $\G_2$-structure on the seven-dimensional Lie algebra $\g=\mathfrak{b}\op \bR$ such that $\mathfrak{b}$ and $\bR$ are orthogonal to each other. Conversely, a cocalibrated $\G_2$-structure on a seven-dimensional Lie algebra $\g=\mathfrak{b}\op \bR$ for which $\mathfrak{b}$ and $\bR$ are orthogonal induces a half-flat $\SU(3)$-structure on $\mathfrak{b}$. So far, there seems to be no example known in the literature of a seven-dimensional Lie algebra $\g=\mathfrak{b}\op \bR$ which admits a cocalibrated $\G_2$-structure such that $\mathfrak{b}$ does not admit a half-flat $\SU(3)$-structure. But now Theorem \ref{th:5d+r2} provides us with an example. Namely $\g=A_{4,5}^{-1/2,-1/2}\op \mathfrak{r}_2\op \bR$ admits a cocalibrated $\G_2$-structure due to Theorem \ref{th:5d+r2} but in \cite{FS} it is shown that $\mathfrak{b}=A_{4,5}^{-1/2,-1/2}\op \mathfrak{r}_2$ does not admit a half-flat $\SU(3)$-structure.
\end{remark}
The only unimodular four-dimensional Lie algebras which are not almost Abelian are the two non-solvable ones $\so(3)\op \bR$, $\so(2,1)\op \bR$ and the two whose commutator ideal $\uf$ is isomorphic to $\h_3$, namely $A_{4,8},\,A_{4,10}$. Direct sums with the non-solvable four-dimensional Lie algebras admit cocalibrated $\G_2$-structures by Corollary \ref{co:nonsolvable}. Direct sums with $A_{4,8},\,A_{4,10}$ admit cocalibrated $\G_2$-structures by Proposition \ref{pro:existence} (c) if $h^1(\g_3)\geq 1$ (note that $h^1(\g_4)-h^2(\g_4)=1$ for $\g_4\in\{A_{4,8},A_{4,10}\}$ by Table \ref{table:4d}). This finishes the proof of Theorem \ref{th:main} (c).
\subsection{$\g_4$ not unimodular, $\g_3$ not unimodular}\label{subsec:fourthcase}
In this subsection, we prove Theorem \ref{th:main} (d). In the following, $\g=\g_4\op \g_3$ always denotes a seven-dimensional Lie algebra which is the Lie algebra direct sum of a four-dimensional non-unimodular Lie algebra $\g_4$ and of a three-dimensional non-unimodular Lie algebra $\g_3$. Furthermore, $\uf$ should always denote the unimodular kernel of $\g_4$

By Proposition \ref{pro:obstruct3dnotunimod} (a), $\g$ does not admit a cocalibrated $\G_2$-structure if $\g_4$ is almost Abelian, i.e. if $\uf$ is Abelian. If $\uf\in \{e(2),e(1,1)\}$, then $\g_4\in \{A_{4,12},\mathfrak{r}_2\op \mathfrak{r}_2\}$ and Proposition \ref{pro:existence} (b) and Lemma \ref{le:contact} imply that $\g$ admits a cocalibrated $\G_2$-structure unless $\g_3=\mathfrak{r}_{3,1}$. But for $\g=A_{4,12}\op \mathfrak{r}_{3,1}$ and $\g=\mathfrak{r}_2\op \mathfrak{r}_2\op\mathfrak{r}_{3,1}$ cocalibrated $\G_2$-structures can be found in Table \ref{table:examples}.

Therefore, it remains to consider the case when the unimodular ideal $\uf$ is isomorphic to $\h_3$.  Then Lemma \ref{le:4dbasis} tells us that we may decompose $\g_4^*=\spa{e^1}\op V_2\op \spa{e^4}$ with $e^1,e^4\neq 0$ and $\dim{V_2}=2$ such that $de^1=\tr(F)e^{14}+\nu$ for $0\neq\nu\in \L^2 V_2$, such that for all $\alpha\in V_2$ the identity $d\alpha=F(\alpha)\wedge e^4$ for some linear map $F:V_2\rightarrow V_2$ with $\tr(F)\neq 0$ is true and such that $de^4=0$. Moreover, by Lemma \ref{le:3d}, we may decompose $\g_3^*=W_2\op \spa{e^7}$ with $0\neq e^7$ and $W_2$ two-dimensional such that for all $\beta\in W_2$ the identity $d\beta=G(\beta)\wedge e^7$ for some linear map $G:W_2\rightarrow W_2$ with $\tr(G)\neq 0$ is true and such that $de^7=0$.
\begin{proposition}\label{pro:technicalities}
Let $\g$, $\g_4$, $\g_3$, $\uf$, $e^1,\,e^4\in \g_4^*\backslash \{0\}$, $e^7\in \g_3^*\backslash \{0\}$, $V_2\subseteq \g_4^*$, $W_2\subseteq \g_3^*$ and $\nu\in \L^2 V_2$ as above. Then $\g$ admits a cocalibrated $\G_2$-structure if and only if there are two linearly independent two-forms $\omega_1,\omega_2\in  V_2\wedge  W_2$, a non-zero two-form $\hat{\nu}\in \L^2 W_2$ and some $\lambda\in \bR$ such that the following conditions are fulfilled:
\begin{enumerate}
\item[(i)]
$d(\omega_1\wedge e^{71}+\omega_2\wedge e^{41})=0$.
\item[(ii)]
The two-forms $\tilde{\omega}_1:=\hat{\nu}+\omega_1$, $\tilde{\omega}_2:=\frac{\tr(F)}{\tr(G)}\hat{\nu}+\lambda\nu+\omega_2$ are linearly independent and each non-zero linear combination is of length two.
\end{enumerate}
\end{proposition}
\begin{proof}
``$\Rightarrow$'':\\
We set
\begin{equation*}
\Lambda^{i,j,k,l}:=\L^i V_2\wedge \L^j W_2\wedge \L^k \spa{e^4}\wedge \L^l \spa{e^7}
\end{equation*}
and denote, for an $s:=(i+j+k+l)$-form $\Phi\in \L^s (V_2\op \spa{e^4}\op \g_3^*)$, by $\Phi^{i,j,k,l}$ the projection of $\Phi$ into $\L^{i,j,k,l}$. Then we have 
\begin{equation*}
d(\L^{i,j,0,0})\subseteq \L^{i,j,1,0}+\L^{i,j,0,1},\,\, d(\L^{i,j,1,0})\subseteq \L^{i,j,1,1},\,\, d(\L^{i,j,0,1})\subseteq \L^{i,j,1,1},\,\, d(\L^{i,j,1,1})=\{0\}
\end{equation*}
for all $i,j\in \mathbb{N}_0$. Moreover, $d(\hat\mu)=-\tr(F)\hat\mu \wedge e^4$ for all $\hat\mu\in \L^{2,0,0,0}$ and $d(\tilde\mu)=-\tr(G)\tilde\mu \wedge e^7$ for all $\tilde\mu\in \L^{0,2,0,0}$.

Let $\Psi\in \L^4 (\g_4\op \g_3)^*$ be the Hodge dual of a cocalibrated $\G_2$-structure. Decompose $\Psi$ into
\begin{equation*}
\Psi=\Omega+e^1\wedge\rho
\end{equation*}
with $\Omega\in \L^4 (V_2\op\spa{e^4}\op \g_3^*)$, $\rho\in \L^3 (V_2\op \spa{e^4}\op \g_3^*)$.
Then
\begin{equation}\label{eq:technicalities}
0=d\Psi=d\Omega+(\tr(F) e^{14}+\nu)\wedge \rho-e^1\wedge d\rho= e^1\wedge (\tr(F)e^4\wedge \rho-d\rho)+d\Omega+\nu\wedge \rho
\end{equation}
implies $\Phi:=\tr(F) e^4\wedge \rho-d\rho=0$. We look at different components of $\Phi$. We have the identities
\begin{equation*}
\begin{split}
0&=\Phi^{2,1,1,0}=\tr(F) e^4\wedge \rho^{2,1,0,0}-d(\rho^{2,1,0,0})^{2,1,1,0}=\tr(F)e^4\wedge \rho^{2,1,0,0}- \tr(F) \rho^{2,1,0,0}\wedge e^4, \\
 &=2\tr(F)e^4\wedge \rho^{2,1,0,0},\\
0&=\Phi^{1,2,0,1}=-d(\rho^{1,2,0,0})^{1,2,0,1}=-\tr(G) \rho^{1,2,0,0}\wedge e^7,\\
0&=\Phi^{2,0,1,1}=\tr(F) e^4\wedge \rho^{2,0,0,1}-d(\rho^{2,0,0,1})=2\tr(F) e^4\wedge \rho^{2,0,0,1},
\end{split}
\end{equation*}
which imply $\rho^{2,1,0,0}=\rho^{1,2,0,0}=\rho^{2,0,0,1}=0$. Moreover,
\begin{equation*}
0=\Phi^{0,2,1,1}=\tr(F) e^4\wedge \rho^{0,2,0,1}-d(\rho^{0,2,1,0})= \tr(F) e^4\wedge \rho^{0,2,0,1}+ \tr(G) e^7\wedge \rho^{0,2,1,0},
\end{equation*}
i.e. $\frac{\tr(F)}{\tr(G)}e^4\wedge \rho^{0,2,0,1}= -e^7\wedge \rho^{0,2,1,0}$. Thus, $\rho$ decomposes as
\begin{equation*}
\rho= e^{7}\wedge(\omega_1+\hat\nu)+e^{4}\wedge \left(\omega_2+\frac{\tr(F)}{\tr(G)}\hat\nu+\lambda \nu\right)+e^{47}\wedge \alpha
\end{equation*}
with $\omega_1,\, \omega_2\in \L^{1,1,0,0}$, $\hat\nu\in \L^{0,2,0,0}$, $\lambda\in \bR$, $\alpha\in \L^{1,0,0,0}\op \L^{0,1,0,0}$. Lemma \ref{le:alginv1} and Lemma \ref{le:alginv2} imply that $\tilde\omega_1:=\omega_1+\hat\nu$ and
$\tilde\omega_2:=\omega_2+\frac{\tr(F)}{\tr(G)}\hat\nu+\lambda \nu$ span a two-dimensional subspace in which each non-zero element is of length two. Moreover,
\begin{equation*}
0=\Phi^{1,1,1,1}=\tr(F) e^4\wedge \rho^{1,1,0,1}-d(\rho^{1,1,1,0})-d(\rho^{1,1,0,1})
\end{equation*}
which shows that
\begin{equation*}
\begin{split}
d(e^1\wedge (\rho^{1,1,1,0}+\rho^{1,1,0,1}))=&(\nu+\tr(F) e^{14})\wedge (\rho^{1,1,1,0}+\rho^{1,1,0,1})-e^1\wedge d(\rho^{1,1,1,0}+\rho^{1,1,0,1})\\
=&\tr(F) e^{14}\wedge \rho^{1,1,0,1}-e^1\wedge d(\rho^{1,1,1,0})-e^1\wedge d(\rho^{1,1,0,1})\\
=&e^1\wedge \Phi^{1,1,1,1}= 0.
\end{split}
\end{equation*}
Since $\rho^{1,1,1,0}=e^4\wedge \omega_2$ and $\rho^{1,1,0,1}=e^7\wedge \omega_1$, we get $d(\omega_1\wedge e^{71}+\omega_2\wedge e^{41})=0$.

What is left to show is that $\hat\nu\neq 0$. Therefore, let $\tilde\Omega$ be the projection of $\Psi$ onto the subspace $\L^4 (\spa{e^1}\op V_2\op W_2)$ (along $\sum_{i=1}^2 \L^i (\spa{e^1}\op V_2\op W_2)\wedge \L^{2-i} \spa{e^4,e^7})$. By Lemma \ref{le:alginv1}, $l(\tilde{\Omega})\geq 1$, i.e. $\tilde{\Omega}\neq 0$. We may write $\tilde\Omega$ in terms of the components of $\rho$ and $\Omega$ as
\begin{equation*}
\tilde\Omega=e^1\wedge\rho^{2,1,0,0}+e^1\wedge \rho^{1,2,0,0}+\Omega^{2,2,0,0}=\Omega^{2,2,0,0}
\end{equation*}
and get $\Omega^{2,2,0,0}\neq 0$. Equation (\ref{eq:technicalities}) gives us
\begin{equation*}
0=(d\Omega+\nu\wedge \rho)^{2,2,0,1}=d\left(\Omega^{2,2,0,0}\right)^{2,2,0,1}+\nu\wedge \rho^{0,2,0,1}=-\tr(G) \Omega^{2,2,0,0}\wedge e^7+\nu\wedge \rho^{0,2,0,1}
\end{equation*}
and so $e^7\wedge\hat\nu =\rho^{0,2,0,1}\neq 0$, i.e. $\hat\nu\neq 0$.

``$\Leftarrow$'':\\
Assume that there exist two-forms $\omega_1$, $\omega_2\in  V_2\wedge  W_2$, $0\neq \hat{\nu}\in \L^2 W_2$ and $\lambda\in \bR$ fulfilling all the conditions. Then $\tilde{\omega}_1$ as in the statement fulfils $0\neq \tilde{\omega}_1^2\in \L^2 V_2\wedge \L^2 W_2$. Hence, there exists $0\neq\tilde{\lambda}\in \bR$ such that $\frac{\tilde{\lambda}}{2} \tilde{\omega}_1^2= -\frac{1}{\tr(G)}\nu \wedge \hat{\nu}$. Set now $\theta_1:=\frac{1}{\tilde\lambda} e^{71}$, $\theta_2:=\frac{1}{\tilde\lambda}e^{41}$, $\theta_3:=e^{74}\in \L^2 \spa{e^1,e^4,e^7}$. By assumption, $\tilde{\omega}_1$, $\tilde{\omega}_2$ as in the statement span a two-dimensional space in which each non-zero element has length two. Thus, we may apply Lemma \ref{le:subspaceslengthtwo} and Proposition \ref{pro:Hodgedualbytwoforms} to $V_4^*:=V_2\op W_2$, $V_3^*:=\spa{e^1,e^4,e^7}$ and get the existence of a two-form $\tilde\omega_3\in \L^2 V_4^*$ such that
\begin{equation*}
\Psi:=\sum_{i=1}^3 \tilde\omega_i\wedge \theta_i+\frac{1}{2}\tilde\omega_1^2
\end{equation*}
is the Hodge dual of a $\G_2$-structure. Using $d\nu=-\tr(F) \nu\wedge e^4$, $d\hat{\nu}=-\tr(G)\hat{\nu}\wedge e^7$, we compute
\begin{equation*}
\begin{split}
d\Psi&=\frac{1}{\tilde\lambda}d \left(\tilde{\omega}_1\wedge e^{71}+\tilde{\omega}_2\wedge e^{41}\right)+d\left(\tilde\omega_3\wedge e^{74}\right)-\frac{1}{\tilde\lambda\cdot \tr(G)}d(\nu\wedge\hat\nu)\\
&=\frac{1}{\tilde\lambda}d\left(\omega_1\wedge e^{71}+\omega_2\wedge e^{41}\right)+\frac{1}{\tilde\lambda}d\left(\hat\nu\wedge e^{71}\right)+\frac{1}{\tilde\lambda}d\left(\frac{\tr(F)}{\tr(G)}\hat\nu\wedge e^{41}+\lambda \nu\wedge e^{41}\right)\\
&\,\,\,\,\,\,+\frac{\tr(F)}{\tilde\lambda\cdot \tr(G)} \nu\wedge\hat\nu\wedge e^4+\frac{1}{\tilde\lambda} \nu\wedge\hat\nu\wedge e^7\\
&= 0-\frac{\tr(F)}{\tilde\lambda} \hat\nu \wedge e^{714}-\frac{1}{\tilde\lambda}\hat\nu\wedge e^7\wedge \nu-\frac{\tr(F)}{\tilde\lambda} \hat\nu\wedge e^{741}-\frac{\tr(F)}{\tilde\lambda\cdot\tr(G)} \hat\nu\wedge e^4\wedge \nu\\
&\,\,\,\,\,\,+\frac{\tr(F)}{\tilde\lambda\cdot \tr(G)} \nu\wedge\hat\nu\wedge e^4+\frac{1}{\tilde\lambda} \nu\wedge\hat\nu\wedge e^7\\
&= 0.
\end{split}
\end{equation*}

\end{proof}
\begin{remark}\label{re:technicalities}
The two-form $\omega_1\in  V_2\wedge  W_2$ in Proposition \ref{pro:technicalities} has to be of length two since $\tilde\omega_1=\omega_1+\hat\nu$ is of length two. By Lemma \ref{le:subspaceslengthtwo}, there exists a basis $e^2,e^3$ of $V_2$ and a basis $e^5,e^6$ of $W_2$ such that $\omega_1=e^{26}+e^{35}$. If $\det(G)\neq 0$, then the condition $d(\omega_1\wedge e^{71}+\omega_2\wedge e^{41})=0$ implies that $\omega_2= (F+\tr(F)\id)(e^2)\wedge G^{-1}(e^6)+(F+\tr(F)\id)(e^3)\wedge G^{-1}(e^5)$.
\end{remark}
Let us, nevertheless, start with $\det(G)=0$.
\begin{lemma}\label{le:det=0}
Let $\g$, $\g_4$, $\g_3$, $e^1,e^4\in  \g_4^*$, $e^7\in \g_3^*$, $V_2$, $F:V_2\rightarrow V_2$, $W_2$ and $G:W_2\rightarrow W_2$ as in Proposition \ref{pro:technicalities}. Assume further that $\det(G)=0$, i.e. $\g_3=\mathfrak{r}_2\op \bR$. Then $\g$ admits a cocalibrated $\G_2$-structure if and only if $\det(F+\tr(F)\id)=0$, i.e. $\g_4=A_{4,9}^{-\frac{1}{2}}$.
\end{lemma}
\begin{proof}
``$\Rightarrow$'':\\
Assume that $\g$ admits a cocalibrated $\G_2$-structure. By Proposition \ref{pro:technicalities} and Remark \ref{re:technicalities}, there exists a basis $e^2,e^3$ of $V_2$ and a basis $e^5,e^6$ of $W_2$ such that $\omega_1:=e^{26}+e^{35}$ fulfils $d(\omega_1\wedge e^{71})\in d( V_2\wedge  W_2\wedge e^{41})= V_2\wedge G(W_2)\wedge e^{741}$. Each element in $ V_2\wedge G(W_2)\wedge e^{741}$ is of length at most one due to $\det(G)=0$. But
\begin{equation*}
d(\omega_1\wedge e^{71})=((F+\tr(F)\id)(e^2)\wedge e^6+(F+\tr(F)\id)(e^3)\wedge e^5)\wedge e^{741}
\end{equation*}
is of length less than two if and only if $\det(F+\tr(F)\id)= 0$. Thus, $\det(F+\tr(F)\id)= 0$.\\
``$\Leftarrow$'':\\
We have $\det(F+\tr(F)\id)=0=\det(G)$ and $\tr(F+\tr(F)\id)=3\tr(F)\neq 0,\tr(G)\neq 0$. Since both $F+\tr(F)\id$ and $G$ are linear endomorphisms in two dimensions, this implies that they diagonalizable over the reals with one zero eigenvalue and one non-zero eigenvalue. We may, after rescaling $e^4$ and $e^7$, assume that the non-zero eigenvalue is equal to one in both cases and so $\tr(F)=\frac{1}{3}$ and $\tr(G)=1$. Since $d(e^1\wedge \alpha)=-e^1\wedge(F+\tr(F)\id)(\alpha)\wedge e^4$ for all $\alpha\in V_2$, there exists a basis $e^2,e^3$ of $V_2$ such that $de^{12}=0$ and $de^{13}=-e^{134}$. Moreover, we may choose a basis $e^5,e^6$ of $W_2$ with $de^5=0$ and $de^6=e^{67}$. Then the following two-forms fulfil all the conditions in Proposition \ref{pro:technicalities}:
\begin{equation*}
\omega_1:=e^{25}-e^{36}+e^{26}, \quad \omega_2:=e^{25}-e^{36}-2 e^{35},\quad \tilde{\omega}_1:=e^{56}+\omega_1,\quad \tilde{\omega}_2:=\frac{1}{3} e^{56}+\omega_2.
\end{equation*}
\end{proof}
If $\det(G)\neq 0$ and $F$ and $G$ are both not multiples of the identity, we get:
\begin{lemma}\label{le:FGnotidentity}
Let $\g$, $\g_4$, $\g_3$, $e^1,e^4\in  \g_4^*$, $e^7\in \g_3^*$, $V_2$, $F:V_2\rightarrow V_2$, $W_2$ and $G:W_2\rightarrow W_2$ as in Proposition \ref{pro:technicalities}. Assume further that $F$ and $G$ are both not multiples of the identity, i.e. $\g_4\neq A_{4,9}^{1}$ and $\g_3\neq \mathfrak{r}_{3,1}$. Then $\g$ admits a cocalibrated $\G_2$-structure.
\end{lemma}
\begin{proof}
Set $H:=-(F+\tr(F)\id)$. Then also $H:V_2\rightarrow V_2$ is not a multiple of the identity, not trace-free and $d(e^1\wedge\alpha)=e^1\wedge H(\alpha)\wedge e^4$ for all $\alpha\in V_2$. By rescaling $e^4$ appropriately, we may assume that $\tr(H)=-3$, i.e. $\tr(F)=1$. Hence, we may choose a basis $e^2,e^3$ of $V_2$ such that the transformation matrix of $H$ with respect to this basis is given by
\begin{equation*}
\begin{pmatrix}
0 & \frac{\det(H)}{\det(G)} \\
-\det(G) & -3
\end{pmatrix}.
\end{equation*}
Moreover, by rescaling $e^7$ appropriately, we may assume that $\tr(G)=1$. Hence, for all $a\in \bR\backslash\{0\}$, we may choose a basis $e^5,e^6$ of $W_2$ such that the transformation matrix of $G$ with respect to this basis is given by
\begin{equation*}
\begin{pmatrix}
0 & -\frac{\det(G)}{a}\\
a & 1
\end{pmatrix}.
\end{equation*}
Set
\begin{equation*}
\omega_1:=e^{25}+e^{36},\; \omega_2:=-\frac{\det(H)}{\det(G) a}e^{25}+\frac{3+a}{a} e^{35}-a\, e^{36},\; \tilde{\omega}_1:=e^{56}+\omega_1,\; \tilde{\omega}_2:=e^{56}-a\, e^{23}+\omega_2.
\end{equation*}
A short computation shows $d(\omega_1\wedge e^{71}+\omega_2\wedge e^{41})=0$. Moreover, $\tilde{\omega}_1^2=2 e^{2536}\neq 0$ and $\tilde{\omega}_1\wedge \tilde{\omega}_2=B\tilde{\omega}_1^2$, $\tilde{\omega}_2^2=C \tilde{\omega}_1^2$ with $B=-\frac{\det(H)}{2 a \det(G)}$ and $C=a+\frac{\det(H)}{\det(G)}$. Hence,
\begin{equation*}
C-B^2=a+\frac{\det(H)}{\det(G)} - \frac{\det(H)^2}{4 a^2\det(G)^2}>0
\end{equation*}
for $a>0$ large enough and so $\tilde{\omega}_1$, $\tilde{\omega}_2$ span a two-dimensional space in which each non-zero element has length two by Lemma \ref{le:subspaceslengthtwo}. Thus, $\g$ admits a cocalibrated $\G_2$-structure by Proposition \ref{pro:technicalities}
\end{proof}
Therefore, it remains to consider the cases when at least one of the maps $F$ and $G$ is (a multiple of) the identity:
\begin{lemma}\label{le:ForGidentity}
Let $\g$, $\g_4$, $\g_3$, $e^1,e^4\in  \g_4^*$, $e^7\in \g_3^*$, $V_2$, $F:V_2\rightarrow V_2$, $W_2$ and $G:W_2\rightarrow W_2$ as in Proposition \ref{pro:technicalities}.
\begin{enumerate}
 \item
If $F$ is a multiple of the identity, i.e. $\g_4=A_{4,9}^1$, then $\g$ admits a cocalibrated $\G_2$-structure if and only if $-\frac{3}{4}\tr(G)^2>\det(G)$ or $\det(G)>0$.
\item
If $G$ is a multiple of the identity, i.e. $\g_3=\mathfrak{r}_{3,1}$, then $\g$ admits a cocalibrated $\G_2$-structure if and only if $\det(F)>-\frac{3}{4}\tr(F)^2$.
\end{enumerate}
\end{lemma}
\begin{remark}
Note that a real two-by-two matrix with negative determinant is always diagonalizable over the reals. The determinant of $G$ is negative if the condition in Lemma \ref{le:ForGidentity} (a) is not fulfilled and the determinant of $F$ is negative if the condition in Lemma \ref{le:ForGidentity} (b) is not fulfilled. Hence, it is easily checked that the condition on $\g_3$ in Lemma \ref{le:ForGidentity} (a) is not fulfilled exactly when $\g_3\in \left\{\mathfrak{r}_{3,\mu}\left|\mu\in \left[-1/3,0\right)\right.\right\}$ and that the condition on $\g_4$ in Lemma \ref{le:ForGidentity} (b) is not fulfilled exactly when $\g_4\in \left\{A_{4,9}^{\alpha}\left|\alpha\in \left(-1,-1/3\right]\right.\right\}$. Hence, proving Lemma \ref{le:ForGidentity} finishes the proof of Theorem \ref{th:main}.
\end{remark}
\begin{proof}
\begin{enumerate}
\item
By rescaling $e^4$ we may assume that $\tr(F)=2$, i.e. $F=\id$. Hence, Proposition \ref{pro:technicalities} and Remark \ref{re:technicalities} tell us that $\g$ admits a cocalibrated $\G_2$-structure if and only if there exists a basis $e^2,e^3$ of $V_2$, a basis $e^5,e^6$ of $W_2$, $\lambda,\alpha\in \bR$, $\alpha\neq 0$ such that each non-zero linear combination of
\begin{equation*}
\tilde{\omega}_{1,\alpha,\lambda}:=\alpha e^{56}+e^{26}+e^{35},\, \, \tilde{\omega}_{2,\alpha,\lambda}:= \frac{2}{\tr(G)} \alpha e^{56}+\lambda e^{23}+3 e^2\wedge G^{-1}(e^6)+3 e^3\wedge G^{-1}(e^5)
\end{equation*}
is of length two. A short computation shows
\begin{equation*}
\begin{split}
\tilde{\omega}_{1,\alpha,\lambda}^2&=2e^{2356},\quad  \tilde{\omega}_{1,\alpha,\lambda}\wedge \tilde{\omega}_{2,\alpha,\lambda}=\left(\alpha \lambda+\frac{3\tr(G)}{\det(G)}\right)e^{2356},\\
\tilde{\omega}_{2,\alpha,\lambda}^2&=\left(4\frac{\alpha \lambda}{\tr(G)}+18\frac{1}{\det(G)}\right)e^{2356}
\end{split}
\end{equation*}
since for an invertible two-by-two matrix $\tr\left(G^{-1}\right)=\frac{\tr(G)}{\det(G)}$. Set $X:=\alpha\lambda$. Then Lemma \ref{le:subspaceslengthtwo} tells us that each non-zero linear combination of $\tilde\omega_{1,\alpha,\lambda}$ and $\tilde\omega_{2,\alpha,\lambda}$ is of length two if and only if the quadratic polynomial
\begin{equation*}
\begin{split}
&8\frac{X}{\tr(G)}+36\frac{1}{\det(G)}-\left(X+\frac{3\tr(G)}{\det(G)}\right)^2\\
&=-X^2+\left(\frac{8 }{\tr(G)}-6 \frac{\tr(G)}{\det(G)}\right)X+ 36\frac{1}{\det(G)}-9 \frac{\tr(G)^2}{\det(G)^2}
\end{split}
\end{equation*}
in $X$ with leading negative coefficient is positive for some $X\in \bR$. Note that this expression does not depend on the basis we have chosen. Hence, $\g$ admits a cocalibrated $\G_2$-structure if and only if this quadratic polynomial is positive for some $X\in \bR$ and this is true if and only if its discriminant is positive. The discriminant is given by
\begin{equation*}
\left(6 \frac{\tr(G)}{\det(G)}-8 \frac{1}{\tr(G)}\right)^2-4\cdot\left(9 \frac{\tr(G)^2}{\det(G)^2}-36\frac{1}{\det(G)}\right)= \frac{16( 3\tr(G)^2+4\det(G) )}{\det(G) \tr(G)^2},
\end{equation*}
and it is positive if and only if
\begin{equation*}
-\frac{3}{4}\tr(G)^2>\det(G) \quad \textrm{ or } \quad \det(G)>0.
\end{equation*}
\item
By rescaling $e^7$ we may assume $\tr(G)=2$, i.e. $G=\id$. Then we see similarly as in the proof of part (a) that $\g$ admits a cocalibrated $\G_2$-structure if and only if there exists a basis $e^2,e^3$ of $V_2$, a basis $e^5,e^6$ of $W_2$, $\lambda,\alpha\in \bR$, $\alpha\neq 0$ such that each non-zero linear combination of
\begin{equation*}
\begin{split}
\tilde{\omega}_{1,\alpha,\lambda}&:=\alpha e^{56}+e^{26}+e^{35},\\
\tilde{\omega}_{2,\alpha,\lambda}&:= \frac{\tr(F)}{2} \alpha e^{56}+\lambda e^{23}+ (F+\tr(F)\id)(e^2)\wedge e^6+(F+\tr(F)\id)(e^3)\wedge e^5
\end{split}
\end{equation*}
is of length two. If we set $X:=\alpha \lambda$ as before, we find, analogously to the proof of (a), that the existence of a cocalibrated $\G_2$-structure on $\g$ is equivalent to the existence of $X\in  \bR$ such that $-X^2-4 \tr(F) X-\tr(F)^2+4\det(F)$ is positive. Note therefore that for a two-by-two matrix $A\in \bR^{2\times 2}$ we generally have $\det(A+\tr(A) I_2)=\det(A)+2\tr(A)^2$. Now $-X^2-4 \tr(F) X-\tr(F)^2+4\det(F)$ is positive for some $X\in \bR$ exactly when the discriminant of this quadratic polynomial in $X$, which is given by $12 \tr(F)^2+16 \det(F)$,
is positive. And this is the case if and only if
\begin{equation*}
\det(F)>-\frac{3}{4}\tr(F)^2.
\end{equation*}
\end{enumerate}
\end{proof}
\newpage
\section*{Acknowledgments}
\pagestyle{myheadings}
\markright{APPENDIX \hfill}
The author thanks the University of Hamburg for financial support and Vicente Cort\'es for many helpful discussions. This work was supported by the
German Research Foundation (DFG) within the Collaborative Research Center 676 ``Particles, Strings and the
Early Universe''.
\section*{Appendix}
Table \ref{table:3d} contains all three-dimensional Lie algebras. The list is further subdivided into the unimodular and the non-unimodular three-dimensional Lie algebras. The names for the non-unimodular Lie algebras in the first column have been adopted from \cite{GOV}. In the second column the Lie bracket is encoded dually. Thereby, $e^5,\, e^6, \,e^7$ is a basis of $\g^*$ and we write down the vector $(de^5,de^6,de^7)$ and use the abbreviation $e^{ij}:=e^i\wedge e^j$. Note that, instead of the more natural denotation of the basis of $\g^*$ by $e^1,\,e^2,\,e^3$, we denote it by $e^5,\, e^6,\,e^7$ since these one-forms are always the last three basis elements in the dual basis of the seven-dimensional Lie algebras we consider. In the last column the vector $(h^1(\g),h^2(\g),h^3(\g))$ of the dimensions of the corresponding Lie algebra cohomology groups is given. We omitted $h^0(\g)$ since it is always equal one.

Table \ref{table:4d} contains all four-dimensional Lie algebras and it is, as before, further subdived into the unimodular and the non-unimodular ones. The names for the Lie algebras in the first column have been adopted from \cite{PSWZ}. In the second column the Lie bracket is encoded dually for a basis $e^1,\,e^2,\, e^3,\, e^4$ of $\g^*$ as in Table \ref{table:3d}. The next column contains the vector $(h^1(\g),h^2(\g),h^3(\g),h^4(\g))$ of the dimensions of the corresponding Lie algebra cohomology groups, where we again omit $h^0(\g)=1$.
The column labelled ``$\uf$'' contains all isomorphism classes of unimodular codimension one ideals in $\g$. Note that for $\bR^4$ there are obviously different codimension Abelian ideals and also for $\h_3\oplus \bR$ there are different codimension one Abelian ideals. The next column, labelled $[\g,\g]$ contains the commutator ideal of $\g$. Finally, in the last column the number $h^1(\g)+h^1(\uf)-h^2(\g)$ is computed. If there is more than one isomorphism class of codimension one unimodular ideals $\uf$, then the different numbers are written next to each other, ordered according to the order in the column ``$\uf$''. 

Table \ref{table:examples} contains (the dual bases of) adapted bases for cocalibrated $\G_2$-structures on three different seven-dimensional Lie algebras $\g$ which are Lie algebra direct sums of a four and a three-dimensional Lie algebra. These three cases are exceptional in the sense that they do not fulfill any of the different conditions we obtained in this article which ensure the existence of a cocalibrated $\G_2$-structure.
\begin{small}
\renewcommand{\arraystretch}{1.5} 
\setlength{\tabcolsep}{0.1cm}
\begin{longtable}[ht]{llc}
  \caption{Three-dimensional Lie algebras} \\ \hline
  $\g$ & Lie bracket &  $\mathrm{h}^*(\g)$ \\
  \hline
  \endhead
  \label{table:3d}
  & unimodular & \\ \hline
  $\so(3)$ & $(e^{67},-e^{57},e^{56})$ & $(0,0,1)$  \\
  $\so(2,1)$ & $(e^{67},e^{57},e^{56})$ & $(0,0,1)$  \\
  $e(2)$ & $(e^{67},-e^{57},0)$ & $(1,1,1)$  \\
  $e(1,1)$ & $(e^{67},e^{57},0)$ & $(1,1,1)$ \\
  $\h_3$ & $(e^{67},0,0)$ & $(2,2,1)$  \\
  $\bR^3$ & $(0,0,0)$ & $(3,3,1)$ \\
  \hline  & non-unimodular & \\ \hline
  $\mathfrak{r}_2\op \bR$ & $(e^{57},0,0)$ & $(2,1,0)$ \\
  $\mathfrak{r}_3$ & $(e^{57}+e^{67},e^{67},0)$  & $(1,0,0)$ \\
  $\mathfrak{r}_{3,\mu}$ & $(e^{57},\mu e^{67},0)$, $-1<\mu\leq 1$, $\mu\neq 0$ & $(1,0,0)$ \\
  $\mathfrak{r}'_{3,\mu}$ &  $(\mu e^{57}+e^{67}, \mu e^{67}-e^{57},0)$, $\mu>0$ & $(1,0,0)$
\end{longtable}
\begin{longtable}[ht]{L{1.7cm}L{5.5cm}cccc}
  \caption{Four-dimensional Lie algebras} \\ \hline
  $\g$ & Lie bracket  & $\mathrm{h}^*(\g)$ &  $\uf$ &  $[\g,\g]$ & $h^1(\g)+h^1(\uf)-h^2(\g)$ \\
  \hline
  \endhead
   \label{table:4d}
  & \multicolumn{2}{c}{unimodular} & & & \\ \hline
  $\so(3)\op \bR$ & $(e^{23},-e^{13},e^{12},0)$ & $(1,0,1,1)$ & $\so(3)$ & $\so(3)$ & $1$  \\
  $\so(2,1)\op \bR$ & $(e^{23},e^{13},e^{12},0)$ & $(1,0,1,1)$ & $\so(2,1)$ &  $\so(2,1)$ & $1$  \\
  $e(2)\op \bR$ & $(e^{23},-e^{13},0,0)$ & $(2,2,2,1)$ & $\bR^3,\,e(2)$ & $\bR^2$ & $3,\,1$  \\
  $e(1,1)\op \bR$ & $(e^{23},e^{13},0,0)$ & $(2,2,2,1)$ & $\bR^3,\,e(1,1)$ & $\bR^2$ & $3,\, 1$ \\
  $\h_3\op \bR$ & $(e^{23},0,0,0)$ & $(3,4,3,1)$ & $\bR^3$, $\h_3$ & $\bR$ & $2,\, 1$ \\
  $\bR^4$ & $(0,0,0,0)$ & $(4,6,4,1)$ & $\bR^3$ & $\{0\}$ & $1$ \\
  $A_{4,1}$ & $(e^{24},e^{34},0,0)$ & $(2,2,2,1)$ & $\bR^3,\,\h_3$ & $\bR^2$ & $3,\,2$ \\
  $A_{4,2}^{-2}$ & $(-2 e^{14},e^{24}+e^{34},e^{34},0)$ & $(1,0,1,1)$ & $\bR^3$ & $\bR^3$ &  $4$ \\
  $A_{4,5}^{\alpha,-(\alpha+1)}$ & $( e^{14},\alpha e^{24},-(\alpha+1) e^{34},0)$, $-1<\alpha\leq -1/2$  & $(1,0,1,1)$ & $\bR^3$ & $\bR^3$ & $4$ \\
  $A_{4,6}^{\alpha,-\alpha/2}$ & $\left( \alpha e^{14},-\frac{\alpha}{2} e^{24}+e^{34},\right.$ $\left.-\frac{\alpha}{2} e^{34}-e^{24},0\right)$,  $\alpha>0$ & $(1,0,1,1)$ & $\bR^3$ & $\bR^3$ & $4$ \\
  $A_{4,8}$ & $(e^{23},e^{24},-e^{34},0)$ & $(1,0,1,1)$ & $\h_3$ & $\h_3$ & $3$ \\
  $A_{4,10}$ & $(e^{23},e^{34},-e^{24},0)$ & $(1,0,1,1)$ & $\h_3$ & $\h_3$ & $3$ \\
 \hline & \multicolumn{2}{c}{non-unimodular} & & \\ \hline
  $\mathfrak{r}_2\op \bR^2$ & $(e^{14},0,0,0)$  &  $(3,3,1,0)$ & $\bR^3$ & $\bR$ & $3$\\
  $\mathfrak{r}_3\op \bR$ & $(e^{14}+e^{24},e^{24},0,0)$  &  $(2,1,0,0)$ & $\bR^3$ & $\bR^2$ & $4$ \\
  $\mathfrak{r}_{3,\mu}\op \bR$ & $(e^{14},\mu e^{24},0,0)$, $-1<\mu\leq 1$, $\mu\neq 0$ &  $(2,1,0,0)$ & $\bR^3$ & $\bR^2$ & $4$ \\
  $\mathfrak{r}'_{3,\mu}\op \bR$ & $(\mu e^{14}+e^{24},-e^{14}+\mu e^{24},0,0)$, $\mu>0$ &  $(2,1,0,0)$ & $\bR^3$ & $\bR^2$ & $4$ \\
   $A_{4,2}^{\alpha}$ & $(\alpha e^{14},e^{24}+e^{34},e^{34},0)$ & & & &  \\
    & $\alpha\neq 0,-1,-2$ & $(1,0,0,0)$ & $\bR^3$ & $\bR^3$ & $4$ \\
   & $\alpha=-1$ & $(1,1,1,0)$ & $\bR^3$ & $\bR^3$ & $3$ \\
   $A_{4,3}$ & $( e^{14},e^{34},0,0)$ & $(2,2,1,0)$ & $\bR^3$ & $\bR^2$ & $3$ \\
   $A_{4,4}$ & $(e^{14}+e^{24},e^{24}+e^{34},e^{34},0)$ & $(1,0,0,0)$ & $\bR^3$ & $\bR^3$ & $4$ \\
   $A_{4,5}^{\alpha,\beta}$ & $( e^{14},\alpha e^{24},\beta e^{34},0)$ &  &  &  & \\
    & $-1<\alpha\leq \beta\leq 1$, $\alpha\beta\neq 0$, $\beta\neq -\alpha,-\alpha-1$ & $(1,0,0,0)$ & $\bR^3$ &  $\bR^3$ & $4$ \\
    & $\alpha=-1$, $\beta>0$, $\beta \neq 1$ & $(1,1,1,0)$ & $\bR^3$ &  $\bR^3$ & $3$ \\
    & $\alpha=-1$, $\beta=1$ & $(1,2,2,0)$ & $\bR^3$ &  $\bR^3$ & $2$ \\
   $A_{4,6}^{\alpha,\beta}$ & $( \alpha e^{14},\beta e^{24}+e^{34},\beta e^{34}-e^{24},0)$ & &  & &  \\
     & $\alpha>0$, $\beta\neq 0,-\alpha/2$ & $(1,0,0,0)$ & $\bR^3$ &  $\bR^3$ & $4$ \\
    & $\beta=0$, $\alpha>0$ & $(1,1,1,0)$ & $\bR^3$ &  $\bR^3$ & $3$ \\
    $A_{4,7}$ & $( 2 e^{14}+e^{23},e^{24}+e^{34},e^{34},0)$ & $(1,0,0,0)$ & $\h_3$ & $\h_3$ & $3$ \\
    $A_{4,9}^{\alpha}$ & $((\alpha+1) e^{14}+e^{23},e^{24},\alpha e^{34},0)$ & & &  &  \\
     & $-1<\alpha\leq 1$, $\alpha\neq -\frac{1}{2},0$ & $(1,0,0,0)$ & $\h_3$ & $\h_3$ & $3$ \\
     & $\alpha=-\frac{1}{2}$ & $(1,1,1,0)$ & $\h_3$ & $\h_3$ & $2$ \\
     & $\alpha=0$ & $(2,1,0,0)$ & $\h_3$ & $\bR^2$ & $3$ \\
    $A_{4,11}^{\alpha}$ & $\left(2 \alpha e^{14}+e^{23},\alpha e^{24}+e^{34},\right.$ $\left.\alpha e^{34}-e^{24},0\right)$, $\alpha>0$& $(1,0,0,0)$ & $\h_3$ & $\h_3$ & $3$ \\
     $A_{4,12}$ & $\left( e^{14}+e^{23},e^{24}-e^{13},0,0\right)$ & $(2,1,0,0)$ & $e(2)$ & $\bR^2$ & $2$ \\
      $\mathfrak{r}_2\op \mathfrak{r}_2$ & $\left( e^{14}+e^{23},e^{24}+e^{13},0,0\right)$\footnotemark & $(2,1,0,0)$ & $e(1,1)$ & $\bR^2$ & $2$ \\
      \footnotetext{A relation of the standard basis $f^1,\,f^2,\,f^3,\,f^4$ of $\mathfrak{r}_2^*\op \mathfrak{r}_2^*$ with $(df^1,df^2,df^3,df^4)=(f^{12},0,f^{34},0)$ to our basis $e^1,\,e^2,\,e^3,\,e^4$ is given by $e^1=f^1+f^3,\, e^2=f^1-f^3,\, e^3=\frac{1}{2}\left(f^2-f^4\right), \, e^4=\frac{1}{2}\left(f^2+f^4\right)$.}
\end{longtable}
\begin{longtable}{lc}
  \caption{Dual adapted bases for cocalibrated $\G_2$-structures for some exceptional cases} \\
  \hline Lie algebra & dual adapted basis \footnotemark \\ \hline
  \endfirsthead
  \hline Lie algebra & dual adapted basis \\
  \hline
  \endhead
  \label{table:examples}
  \footnotetext{In each case, $(\e^1,\dots,\e^7)$ denotes a basis such
    that $\e^1,\dots,\e^4$ satisfy the Lie algebra structure given in
    Table \ref{table:4d} and $\e^5,\ldots,e^7$ satisfy the Lie algebra structure given in Table \ref{table:3d}}
   $A_{4,8}\op e(1,1)$ & $\left(\e^5,\e^6,\e^7,\e^4,\e^2,\e^3,\e^1\right)$ \\
$A_{4,12}\op \mathfrak{r}_{3,1}$ & $\left(-\frac{1}{3} \sqrt{5}\, \e^1, \sqrt{5}\, \e^4, \e^2-\frac{4}{5} \sqrt{5}\, \e^5,\e^3+\frac{2}{5} \sqrt{5}\, \e^6,\e^5,\e^6,\e^7\right)$ \\
$\mathfrak{r}_2 \op \mathfrak{r}_2 \op \mathfrak{r}_{3,1}$ & $\left(\e^2+\frac{13}{9} \e^5, \e^5, \e^3+3\e^6, \e^6, \frac{1}{2\sqrt{10}} \e^7, \frac{1}{3\sqrt{10}} \e^4,\frac{9}{\sqrt{10}} \e^1 \right)$ 
\end{longtable}
\end{small}
\newpage
\pagestyle{headings}

\end{document}